\documentclass[a4paper,12pt]{amsart}
\usepackage[english]{babel}
\usepackage[T1]{fontenc}
\usepackage{mathrsfs}
\usepackage{array}
\usepackage{makecell}

\usepackage{hyperref}
\usepackage{cleveref}

\usepackage[a4paper,margin=2.0cm]{geometry}
\usepackage[justification=centering]{caption}

\usepackage{float}
\usepackage{enumerate}
\usepackage{tabularx} 
\usepackage{amsmath} 
\usepackage{graphicx}
\usepackage{amsmath,amscd,amsbsy,amssymb,latexsym,url,bm,amsthm,mathrsfs}
\usepackage{epsfig,graphicx,subfigure}
\usepackage{enumerate,balance,mathtools}
\usepackage{wrapfig}
\usepackage[vlined,boxed,commentsnumbered,linesnumbered,ruled]{algorithm2e}
\usepackage{amsfonts}
\usepackage{amsthm}
\usepackage{fancyhdr}
\usepackage{picinpar}
\usepackage{listings}
\usepackage{layout}
\usepackage[all]{xy}
\usepackage{indentfirst}
\usepackage{tikz-cd}
\usepackage{tikz}
\usepackage[all]{xy}
\usepackage{mathrsfs}
\usepackage{hyperref}
\hypersetup{
colorlinks=true,
linkcolor=[RGB]{0,0,204},
filecolor=magenta,      
urlcolor=cyan,
citecolor=[RGB]{0,204,0},
}
\usepackage{setspace}

\usepackage{amsmath,amsthm,amssymb,amsxtra,calligra,mathrsfs}
\usepackage{graphicx}
\usepackage{tikz}
\usepackage{tikz-cd} 
\usepackage{xcolor}
\usepackage{upgreek}
\usepackage{comment}
\usepackage{graphicx}
\usepackage{mathtools}
\usepackage{extarrows}
\usepackage{faktor}
\usepackage{mathrsfs,euscript}
\usepackage{comment}
\usepackage{bookmark}
\usepackage{mathrsfs}
\usepackage{multirow}

\newcommand{\wt}[1]{\widetilde{#1}}

\newcommand{\mb}[1]{\mathbb{#1}}

\newcommand{\ove}[1]{\overline{#1}}

\newcommand{\mtc}[1]{\mathcal{#1}}
\newcommand{\mtf}[1]{\mathfrak{#1}}
\newcommand{\mts}[1]{\mathscr{#1}}

\DeclareMathOperator{\tw}{tw}

\DeclareMathOperator{\stab}{Stab}
\DeclareMathOperator{\st}{st}
\DeclareMathOperator{\PGL}{PGL}

\DeclareMathOperator{\Spec}{Spec}

\DeclareMathOperator{\Proj}{Proj}

\DeclareMathOperator{\mult}{mult}

\DeclareMathOperator{\CY}{CY}
\DeclareMathOperator{\Supp}{Supp}

\DeclareMathOperator{\Pic}{Pic}

\DeclareMathOperator{\codim}{codim}

\DeclareMathOperator{\Sym}{Sym}

\DeclareMathOperator{\GIT}{GIT}

\DeclareMathOperator{\Aut}{Aut}

\DeclareMathOperator{\vol}{vol}

\DeclareMathOperator{\sm}{sm}

\newcommand{\bF}{\mathbb{F}}

\newcommand{\sF}{\mathscr{F}}

\newcommand{\sheafHom}{\mathscr{H}\text{\kern -3pt {\calligra\large om}}\,}

\DeclareMathOperator{\sing}{\mathrm{sing}}

\DeclareMathOperator{\KSBA}{\mathrm{KSBA}}

\newcommand{\bV}{{\mathbb V}}

\newcommand{\bA}{{\mathbb A}}
\newcommand{\bZ}{{\mathbb Z}}

\newcommand{\bC}{{\mathbb C}}

\newcommand{\bQ}{{\mathbb Q}}
\newcommand{\bP}{{\mathbb P}}
\newcommand{\bR}{{\mathbb R}}

\usepackage{longtable}
\allowdisplaybreaks
\usepackage{xcolor}

\newcommand{\sR}{{\mathscr{R}}}
\newcommand{\sT}{{\mathscr{T}}}
\newcommand{\sY}{{\mathscr{Y}}}

\newcommand{\sP}{{\mathscr{P}}}
\newcommand{\sS}{{\mathscr{S}}}

\newcommand{\sC}{{\mathscr{C}}}

\newcommand{\cX}{{\mathcal{X}}}

\newcommand{\cB}{{\mathcal{B}}}
\newcommand{\cY}{{\mathcal{Y}}}
\newcommand{\cG}{{\mathcal{G}}}

\newcommand{\cO}{{\mathcal{O}}}
\newcommand{\cR}{{\mathcal{R}}}

\newcommand{\cV}{{\mathcal{V}}}
\newcommand{\bmu}{\bm{\mu}}
\newcommand{\oH}{{\operatorname{H}}}

\definecolor{darkgreen}{rgb}{0.0, 0.5, 0.0}

\DeclareMathOperator{\bfM}{{\bf{M}}}

\newcommand\spec{\text{\rm Spec}}

\newcommand{\cF}{{\mathcal F}}

\newcommand\cD{{\mts{D}}}
\newcommand\cS{{\mathcal{S}}}
\newcommand\cZ{{\mathcal{Z}}}

\newcommand\cL{\mathcal{L}}

\newcommand\cC{{\mathcal{C}}}
\newcommand\cM{{\mathcal{M}}}

\newtheorem{theorem}{Theorem}[section]

\newtheorem{lemma}[theorem]{Lemma}

\newtheorem{corollary}[theorem]{Corollary}
\newtheorem{prop}[theorem]{Proposition}
\newtheorem{construction}[theorem]{Construction}

\newtheorem*{notation}{Notation}

\newtheorem{question}[theorem]{Question}

\newtheorem{remark}[theorem]{Remark}

\theoremstyle{definition}
\newtheorem{defn}[theorem]{Definition}

\newtheorem{example}[theorem]{Example}

\theoremstyle{remark}

\title{Moduli of surfaces fibered in (log) Calabi-Yau pairs II: elliptic surfaces}
\date{\today}

\author{Giovanni Inchiostro}
\address{C-138 Padelford, Box 354350, Seattle, WA 98195, USA}
\email{ginchios@uw.edu}

\author{Junyan Zhao}
\address{William E. Kirwan Hall, 4176 Campus Dr, College Park, MD 20742, USA}
\email{jzhao81@umd.edu}

\begin{document}

\begin{abstract}
In this paper, we study the KSBA moduli stack of elliptic surfaces with a bisection. 
We classify the surfaces that appear on the boundary of this stack. 
As an application, we compactify the moduli stack of hyperelliptic K3 surfaces. A key difficulty is that the absence of a section on the elliptic surfaces prevents the use of explicit steps of the minimal model program. 
Instead, we follow the new approach developed in \cite{ISZ25}. 
Using this method, we also give a simple proof of results from a series of papers on the moduli stacks of elliptic surfaces with a section 
\cite{AB_del_pezzo,inchiostro_elliptic,Bru15}.
\end{abstract}

\maketitle

\tableofcontents

\section{Introduction}
Elliptic fibrations play a central role in the study of algebraic surfaces and their moduli. 
When an elliptic surface admits a section, it can be described by a Weierstrass model, and the section provides a distinguished origin on each fiber. 
This additional structure allows one to exploit the group law on the fibers and gives access to powerful tools such as the canonical bundle formula, explicit Weierstrass equations, and concrete descriptions of degenerations. 
Consequently, elliptic surfaces with a section and their moduli have been extensively studied, ranging from the GIT approach \cite{miranda1981moduli}, to twisted stable maps \cite{AV02}, and to KSBA theory \cite{la2002explicit, ab_twisted, inchiostro_elliptic, ascher_invariance_pluri, AB_k3, ABE, BFHIMZ24}.

In contrast, elliptic surfaces without a section are substantially more difficult to analyze. 
The absence of a section complicates the study of degenerations and stable reduction, and makes it difficult to carry out explicit steps of the minimal model program. 
To address this issue, in \cite{ISZ25} we developed a new approach to study the KSBA moduli stack of fibered Calabi--Yau pairs. 
In the present paper, we apply this framework to study the moduli of elliptic surfaces with a bisection. 
We expect that this approach can be applied more broadly to the study of moduli of elliptic surfaces without a section.

\subsection*{Elliptic surfaces with a bisection}
A bisection equips the elliptic surface with an involution whose restriction to each fiber is the elliptic involution; see Construction \ref{construction}. Consider the KSBA moduli of pairs $(X, \epsilon_1 R + \epsilon_2 F)$ where:
\begin{enumerate}
    \item $X$ is a projective surface admitting a fibration $f \colon X \to C$ whose fibers are integral genus-one nodal curves, with general fiber smooth and with singular fibers marked as $F$;
    \item $X$ carries a fiberwise involution whose fixed locus is a horizontal divisor $R$; and
    \item the pair $(X, \epsilon_1 R + \epsilon_2 F)$ is KSBA-stable for $0<\epsilon_1\ll \epsilon_2 \ll 1$.
\end{enumerate}

The key reduction of this paper is to take the quotient $X \to Y$ by the fiberwise involution, which allows us to transfer the KSBA analysis from elliptic surfaces to log Calabi--Yau fibrations. 
The resulting surface $Y$ carries a natural divisor $R_Y \subseteq Y$ recording the ramification, and the induced fibration
\[
f_Y \colon (Y, \tfrac{1}{2} R_Y) \to C
\]
forms a fibered log Calabi--Yau pair. 
This enables us to apply the results of \cite{ISZ25} to study the original pairs $(X, \epsilon_1 R + \epsilon_2 F)$.

We begin by analyzing such pairs in full generality. 
Our main results classify the surfaces that occur on the boundary of the KSBA moduli space generically parametrizing pairs $(X, \epsilon_1 R + \epsilon_2 F)$ as above.

\begin{theorem}\label{thm_intro_1} Let $\ove{\mtc{M}}^{\KSBA}_{\textup{ell},2}(\epsilon_1,\epsilon_2)$ be the closure of the locus of elliptic surfaces as above, and let $(X,\epsilon_1R+\epsilon_2F)$ be a pair appearing on the boundary of this moduli space. Then:
\begin{enumerate}
    \item there is a morphism $ X\to C$ to a nodal curve, with pure 1-dimensional fibers,
    \item the possible singularities of $(X,\epsilon_1R+\epsilon_2F)$ over the smooth locus of $C$ are classified in Theorems \ref{thm:Singularities over the Smooth Locus I} and \ref{thm:Singularities over smooth locus II}, and
    \item the possible fibers of $(X,\epsilon_1R+\epsilon_2F)\to C$ are classified in Theorems \ref{thm:Singularities over the Smooth Locus I}, \ref{thm:Singularities over smooth locus II} and \ref{thm:Singular fibers over nodes}.
\end{enumerate}
\end{theorem}

\subsection*{Moduli of hyperelliptic K3 surfaces}  A polarized K3 surface $(X,L)$ of degree $(L^2)=2g-2$ is called \emph{hyperelliptic} if $L$ is globally generated but not very ample. By \cite{May72}, for a polarized K3 surface $(X,L)$, if $L$ is not very ample, then 
\begin{enumerate}
    \item either $|L|$ is base-point free, and $(X,L)$ is hyperelliptic; or
    \item $|L|$ has a base component, and $(X,L)$ is \emph{unigonal}.
\end{enumerate}
The hyperelliptic and unigonal K3 surfaces form two Noether-Lefschetz divisors in the moduli of polarized K3 surfaces, called the \emph{hyperelliptic divisor} and \emph{unigonal divisor} respectively. As an application, our main results compactify these two divisors, and describe the objects parametrized by the boundaries explicitly. As the moduli of unigonal K3 surfaces is studied in \cite{Bru15,AB_k3}, we only state the result concerning the hyperelliptic K3 surfaces; see Corollary \ref{cor:HE K3 surfaces}.

\begin{theorem}[Compactification of the moduli of hyperelliptic K3 surfaces]\label{thm:compactification of moduli of HE K3}
For any genus $g \geq 3$, there exists a Deligne–Mumford stack $\mtc{M}^{\KSBA}_{\textup{ell},2}(\epsilon_1,\epsilon_2)$, whose normalization of this stack is a compactification of the normalization of the hyperelliptic divisor $\mts{D}^{g}_{2,0}$ in the moduli stack $\mts{F}_g$ of polarized K3 surfaces of genus $g$. If a pair appears on the boundary of $\mtc{M}^{\KSBA}_{\textup{ell},2}(\epsilon_1,\epsilon_2)$, it is described in Theorems~\ref{thm:elliptic with a bisection moduli}, \ref{thm:Singularities over the Smooth Locus I}, \ref{thm:Singularities over smooth locus II}, and~\ref{thm:Singular fibers over nodes}. 
\end{theorem}

\subsection*{Elliptic surfaces with a section}

Revisiting the KSBA moduli of elliptic surfaces with a section, our approach also yields an alternative—and significantly streamlined—proof of certain results in \cite{AB_k3, AB_del_pezzo, inchiostro_elliptic}. 
Let $0<\epsilon\ll 1$ be a sufficiently small rational number, and let 
\[
\vec{a}=(a_1,\ldots,a_n), \qquad 0 < a_i \le 1.
\]
Consider a Weierstrass fibration
\[
(X_\eta,\ \epsilon S_\eta + \vec{a}F_\eta)\ \to\ C_\eta
\]
over the generic point $\eta$ of a DVR $A$, whose fibers are at worst irreducible nodal curves, and such that the pair $(X_\eta,\epsilon S_\eta+\vec{a}F_\eta)$ is KSBA-stable. Here:
\begin{itemize}
    \item $S_\eta$ is a section of $X_\eta\to C_\eta$;
    \item $p_\eta = p_{\eta,1}+\cdots+p_{\eta,n}$ is a divisor of marked points on $C_\eta$, and 
    \[
        \vec{a}p_\eta := \sum_{i=1}^n a_i p_{\eta,i};
    \]
    \item $F_\eta = F_{\eta,1}+\cdots+F_{\eta,n}$ is the corresponding sum of fibers, and 
    \[
        \vec{a}F_\eta := \sum_{i=1}^n a_i F_{\eta,i}.
    \]
\end{itemize}
Let $\mathbf{M}_\eta$ denote the moduli part of the canonical bundle formula for $X_\eta \to C_\eta$.

\smallskip
\begin{theorem}
The following statements hold:
\begin{enumerate}
    \item[\textup{(1)}] After a finite base change of $\Spec A$, the KSBA-stable extension $(X,\epsilon S+\vec{a}F)$ of $(X_\eta,\epsilon S_\eta+\vec{a}F_\eta)$ over $\Spec A$ admits a fibration over a generalized pair
    \[\textstyle
       (X,\epsilon S+\vec{a}F)\xrightarrow{\pi} (C,\vec{a}p+\bfM)\xrightarrow{g}\Spec A,
    \]
    with $g$  a family of nodal pointed curves with $K_C+\vec{a}p+\bfM$ ample, and whose generic fiber is $ (C_\eta,\vec{a}p_\eta+\mathbf{M}_\eta)$, and such that $f^*(K_C+\bfM +\vec{a}p)\sim_\bQ K_X+\vec{a}F$.
    \item[\textup{(2)}] The family $C\to\Spec A$ is the coarse space of a family of twisted curves $\cC \to \Spec A$ equipped with a fibration $(\cX,\cS) \to \cC$, whose geometric fibers $(\cX_p,\cS_p)$ are Weierstrass fibrations of the form
    \[
    \bigl( y^2x = x^3 + axz^2 + bz^3,\ [0,1,0] \bigr),
    \]
    and such that the induced morphism on coarse spaces is $(\cX,\cS)\to\cC$ giving $(X,S)\to C$.

    \item[\textup{(3)}] The fibers of $(X,S)\to C$ over codimension-one points of $C$ have at worst nodal singularities.
\end{enumerate}
\end{theorem}
\noindent The following diagram summarizes the situation:
    \[
    \begin{tikzcd}[ampersand replacement=\&]
        {(X_\eta,\epsilon S_\eta + \vec{a}F_\eta)} \& {(X,\epsilon S+\vec{a}F)} \&\& {(\cX,\epsilon\cS+\vec{a}\cF)} \\
        {(C_\eta,\vec{a}p_\eta+\bfM)} \& {(C,\vec{a}p+\bfM)} \&\& {\cC} \\
        \eta \& {\Spec A}
        \arrow[from=1-1, to=1-2]
        \arrow[from=1-1, to=2-1]
        \arrow[from=1-2, to=2-2]
        \arrow["{\textup{cms}}"', from=1-4, to=1-2]
        \arrow["{\textup{Weierstrass}}"{description}, from=1-4, to=2-4]
        \arrow[from=2-1, to=2-2]
        \arrow[from=2-1, to=3-1]
        \arrow[from=2-2, to=3-2]
        \arrow["{\textup{cms}}"{description}, from=2-4, to=2-2]
        \arrow[from=2-4, to=3-2]
        \arrow[hook, from=3-1, to=3-2]
    \end{tikzcd}
    \]
    
The stable reduction arguments in \cite{AB_k3, AB_del_pezzo, inchiostro_elliptic} rely on an explicit MMP and, in particular, on an explicit flip—the \emph{La Nave flip} \cite{la2002explicit}—which renders parts of the analysis rather technical. 
In contrast, the present work avoids the La Nave flip entirely by invoking general results from \cite{ISZ25}, especially \cite[Theorem~1.2]{ISZ25}, together with the canonical bundle formula.

\subsection*{Outline of the paper}

In Section~\S2, we summarize the results of \cite{ISZ25} that will be used throughout the paper. Section~\S3 analyzes the double covers of the components of KSBA-stable surface pairs fibered in log Calabi--Yau curves over a nodal base, and classifies the fibers over the nodes. We then combine these ingredients to describe the KSBA moduli stack of elliptic surfaces with an unmarked bisection. Section~\S4 systematically studies the hyperelliptic divisor in the moduli stack of polarized K3 surfaces and the corresponding moduli stack of lattice-polarized K3 surfaces, and proves Theorem~\ref{thm:compactification of moduli of HE K3}. Finally, in Section~\S5 we use the results of \cite{ISZ25} to recover the descriptions of the KSBA moduli stacks of elliptic surfaces with a marked section obtained in \cite{AB_del_pezzo, inchiostro_elliptic, Bru15}.

\begin{itemize}
    \item We work over an algebraically closed field $k$ of characteristic $0$. The reader may assume $k=\mathbb{C}$.
    \item We say that $Y \to (X,\tfrac{1}{2}D)$ is a double cover if $Y \to X$ is a double cover branched along the divisor $D$.
    \item The \emph{coarse space} of a Deligne–Mumford (abbreviated DM) stack $\mtc{X}$ is its coarse moduli space $\mtc{X}\to X$ (cf.~\cite[Tag~04UX]{stacks-project}). We drop the word ``moduli'' since many stacks we consider need not have a modular interpretation.
    \item A ``big open'' subset of an algebraic stack $\cX$ is a Zariski open subset containing all the codimension one points.
    
    \item Throughout the paper, we denote by $A$ a DVR, since the letter $R$ is reserved for the ramification divisor. For any DVR $A$, let $\eta$ (resp.\ $0$) be the generic (resp.\ closed) point of $\Spec A$. If $f : X \to \Spec A$ is an object over $\Spec A$ (e.g.\ a family of pairs), we denote by $X_{\eta}$ (resp.\ $X_0$) the generic (resp.\ special) fiber of $f$.
    
    \item Let $A$ be a DVR with generic point $\eta$. For an object $X_{\eta} \to \eta$, an \emph{extension} of $X_{\eta}$ is an object $X \to \Spec A$ such that the fiber product $X \times_{A} \eta$ is isomorphic to $X_{\eta}$. This explains the previous convention. A \emph{limit} of $X_{\eta}$ is the special fiber $X_0$ of an extension $X \to \Spec A$. The extensions (resp.\ limits) of interest are the KSBA–stable extensions (resp.\ limits), which are unique.
\end{itemize}

For the reader’s convenience, we collect here the notations most frequently used throughout the paper. For the first five lines, when we write ``KSBA-moduli stack'' of a certain class of objects $\mathfrak{C}$ (for example, rational Weierstrass fibrations or K3 Weierstrass fibrations), we mean the closure in the KSBA-moduli space of the locus parametrizing objects in $\mathfrak{C}$.
\renewcommand{\arraystretch}{1.5}

\begin{center}
\begin{longtable}{| p{.18\textwidth} | p{.75\textwidth} |}
    \hline \textbf{Notation} & \textbf{Definition/Description}   \\ 
    \hline $\overline{\mtc{M}}^{\KSBA}(\epsilon_1,\epsilon_2)$ &   KSBA-moduli stack of surface pairs fibered in log-canonical pairs of the form $\big(\bP^1,\frac{1}{2}(p_1+\cdots+p_4)\big)$ \\ \hline
     $\ove{\mtc{W}}^{\KSBA}_{\textup{rat}}(\epsilon_1,\epsilon_2)$   & KSBA-moduli stack of rational Weierstrass fibrations \\ \hline
     $\ove{\mtc{W}}^{\KSBA}_{\textup{K3}}(\epsilon_1,\epsilon_2)$   & KSBA-moduli stack of K3 Weierstrass fibrations \\ \hline
     $\ove{\mtc{W}}^{\KSBA}(\epsilon_1,\epsilon_2)$   & KSBA-moduli stack of Weierstrass fibrations with $\kappa= 1$\\ \hline
       $\ove{\mtc{M}}^{\KSBA}_{\textup{ell},2}(\epsilon_1,\epsilon_2)$   & KSBA-moduli stack of elliptic surfaces with a unmarked bi-section   \\ \hline
    $\sP_1^{\GIT}$ & GIT moduli stack $\big[|\mtc{O}_{\bP^1}(4)|^{ss}/\PGL(2)\big]$ of pairs $(\bP^1,p_1+\cdots+p_{4})$     \\ \hline
      $\sP_1$   & an enlargement of $\sP_1^{\GIT}$, see \cite[Section 2.1]{ISZ25}   \\ \hline
       $\Spec A$ & Spectrum of a DVR with generic point $\eta$ and special point $0$   \\ \hline
        $\mts{F}_{\Lambda}$ & moduli stack of lattice $\Lambda$-polarized K3 surfaces \\ \hline
\end{longtable}
    
\end{center}

\subsection*{Acknowledgments}
We thank Dori Bejleri, Yuchen Liu, Luca Schaffler and Roberto Svaldi for helpful conversations. GI was supported by AMS-Simons travel grant and by the NSF grant DMS-2502104. JZ was supported by AMS-Simons travel grant.

\section{KSBA-moduli stacks for certain ruled surfaces}\label{sec:KSBA-moduli stacks for certain ruled surfaces}

In this section, we summarize the structural results and relevant machinery developed in \cite{ISZ25} that will be used in this paper, for the reader's convenience. We refer the reader to \cite[\S1, Introduction]{ISZ25} for a more detailed overview.

\begin{defn}
    A pair $(X, D)$ is called \textit{KSBA-stable} if 
    \begin{enumerate}
        \item it is semi-log canonical (abbv. slc) and $X$ is connected;
        \item $K_X+D$ is an ample $\bQ$-Cartier $\bQ$-divisor.
    \end{enumerate}
    The \emph{volume} of a KSBA-stable pair $(X,D)$ is $\vol(X,D) = (K_X + D)^{\dim X}$.
\end{defn}

\begin{theorem}[\textup{cf. \cite[Theorem 8.15]{Kol23}}]
    Fix $v\in \mb{Q}_{>0}$ and some rational numbers $c_i\in [0,1]_{\mb{Q}}$. Then there is a proper Deligne-Mumford stack $\overline{\mtc{M}}^{\KSBA}_{c_i,v}$, whose closed points parametrize KSBA-stable surface pairs $(X,\sum c_iD_i)$ with volume $(K_X+\sum c_iD_i)^2=v$, where each $D_i$ is an effective $\mathbb{Z}$-divisor.
\end{theorem}

Consider two-dimensional log pairs
$(Y, \frac 12 R)$, where $Y$ is a proper surface and $R$ is a Weil divisor, that are endowed with a fibration 
\[
\pi : \bigl(Y, \tfrac{1}{2}R\bigr) \longrightarrow C
\]
to a curve whose fibers are one-dimensional log Calabi–Yau pairs of the form
\((\mathbb{P}^1, \tfrac{1}{2}(p_1+\cdots+p_{4}))\), satisfying that
\begin{enumerate}
    \item[(a)] 
    the divisor $R$ cuts $4$ distinct points on a general fiber of $\pi$;
    \item[(b)] 
    denoting by $F$ the sum of the fibers of $\pi$ meeting $R$ in fewer than $4$ points, 
    the pair $\bigl(Y, (\tfrac{1}{2}+\epsilon_1)R+\epsilon_2F\bigr)$ is KSBA-stable for every
    $0 < \epsilon_1 \ll \epsilon_2 \ll 1$.
\end{enumerate}

We now recall the main result of \cite{ISZ25} for the case $n=1$ in \textit{loc. cit.}, dividing it into two parts.
\begin{theorem}[\textup{cf. \cite[Theorem 1.1]{ISZ25}}]\label{thm: main theorem of ISZ25} For $0<\epsilon_1\ll \epsilon_2\ll1$, consider the closure of the locus of surface pairs $(Y,(\frac{1}{2}+\epsilon_1)R +\epsilon_2 F)$ satisfying conditions (a)-(b) as above, with all the fibers of $\pi$ being log-canonical, denoted by $\overline{\mtc{M}}^{\KSBA}(\epsilon_1,\epsilon_2)$. Then $\overline{\mtc{M}}^{\KSBA}(\epsilon_1,\epsilon_2)$ parametrizes KSBA-stable pairs $$\textstyle \big(Y_0,(\frac{1}{2}+\epsilon_1)R_0+\epsilon_2F_0\big),$$ where 
\begin{itemize}
    \item $Y_0$ is a demi-normal surface which admits a morphism $g_0:Y_0\rightarrow C_0$ to a (possibly reducible) nodal curve with equi-dimensional fibers;
    \item $R_0$ is a divisor relatively ample over $C_0$ such that $g_0|_{R_0}:R_0\rightarrow C_0$ is of degree $4$, with the map $R_0\to C_0$ possibly not finite; and
    \item $F_0$ is a divisor whose support consists of fibers $F_{0,i}$ of $g_0:Y_0\rightarrow C_0$ such that $R_0|_{F_{0,i}}$ is not four distinct points.
\end{itemize}
\end{theorem}
The map $R_0\to C_0$ is possibly not finite as it could contain some fiber components of $Y_0\to C_0$.
\begin{theorem}[\textup{cf. \cite[Propositions~4.3, 4.4, and~4.6]{ISZ25}}]\label{thm: main theorem 2 of ISZ25}
    With notation as above, let $G$ be an irreducible component of $C_0$, and let $G^n \to G$ denote its normalization. Then there exists a smooth orbifold curve $\cG^n$ with coarse space $G^n$ such that the reduced structure of $Y|_{G^n}$ is the coarse moduli space of one of the following (see \textup{Figure \ref{fig:Surfaces of different types}}):
    \begin{enumerate}
        \item a projective bundle $\bP_{\cG^n}(\cV)$ over $\cG^n$;
        \item a weighted blow-up of the surface $\bP_{\cG^n}(\cV)$;
        \item the gluing of a surface $\bP_\mtc{E}(\cV)$ ruled over a smooth, possibly non-connected orbifold curve $\mtc{E}$ with at most two connected components, glued along an involution $\mtc{E} \to \mtc{E}$ acting transitively on the connected components of $\mtc{E}$.
    \end{enumerate}
\end{theorem}

\begin{figure}
    \centering
    \includegraphics[width=0.7\linewidth]{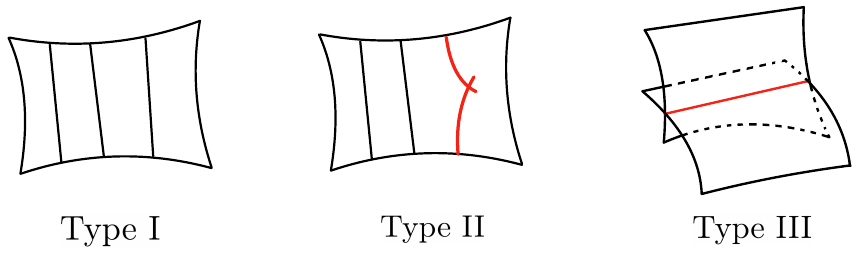}
    \caption{Surfaces of different types}
    \label{fig:Surfaces of different types}
\end{figure}
 
\begin{remark}\label{rem:no difference over smooth locus}\textup{
We refer to surfaces in the classes \textup{(1)}, \textup{(2)}, and \textup{(3)} in Theorem \ref{thm: main theorem 2 of ISZ25} as being of type~I, II, and III, respectively. Surfaces of type~I and type~II differ only over the nodal locus of $C_0$. In particular, a surface of type~II, when restricted to $G \cap C_0^{\sm}$, is the coarse moduli space of a projective bundle as a surface of type I.}
\end{remark}

\begin{corollary}\label{cor:projective bundle over smooth locus}
Let $G$ be a component of $C^{\st}$, and suppose that $Y|_G$ is a surface of type~\textup{I} or~\textup{II}. Then the morphism \(Y|_G \to G\) is a projective bundle over \(G \cap C^{\sm}\).
\end{corollary}

\begin{proof}
It follows from the results in \cite{ISZ25} that for the surface $Y|_G$, the only stacky (i.e.\ nontrivial orbifold) behavior arises over the nodal locus \(C^{\sing}\). Over the smooth locus \(G \cap C^{\sm}\), every fiber is $\bP^1$, and hence it is a projective bundle.
\end{proof}

Before stating the next result, we briefly outline the strategy used to prove the previous results.  
Since the moduli stack $\overline{\mtc{M}}^{\KSBA}(\epsilon_1,\epsilon_2)$ is separated and every object it parametrizes is obtained as a degeneration over the spectrum of a DVR $A$ of a general pair 
\[
(Y_\eta,(\tfrac{1}{2}+\epsilon_1)R_\eta + \epsilon_2 F_\eta),
\]
one may instead start with such a general pair over the generic point $\eta\in\Spec A$ and recover its unique KSBA-stable limit over $\Spec A$. This is carried out in two steps:

\begin{enumerate}
\item Using stable quasimaps \cite{twisted_map_2} one constructs an extension
\[
\bigl(\mtc{Y}^{\tw},(\tfrac{1}{2}+\epsilon_1)\mtc{R}^{\tw}+\epsilon_2\mtc{F}^{\tw}\bigr)\ \longrightarrow\ \Spec A
\]
as a DM stack, together with a morphism to a family of twisted curves 
\[
\mtc{C}^{\tw} \to \Spec A.
\]
This fits into the diagram
\[\begin{tikzcd}[ampersand replacement=\&]
	{\big(\mtc{Y}^{\tw},(\frac{1}{2}+\epsilon_1)\mtc{R}^{\tw}+\epsilon_2\mtc{F}^{\tw}\big)} \&\& {(\mtc{C}^{\tw},\bfM^{\tw})} \\
	\& {\Spec A}
	\arrow["\pi", from=1-1, to=1-3]
	\arrow["f"', from=1-1, to=2-2]
	\arrow["g", from=1-3, to=2-2]
\end{tikzcd}\]
where $\mathbf{M}^{\tw}$ is the moduli part of the canonical bundle formula for the lc-trivial fibration 
\[
(\mtc{Y}^{\tw},\tfrac{1}{2}\mtc{R}^{\tw}) \to \mtc{C}^{\tw}.
\]
Every geometric fiber of $\pi$ is a log Calabi–Yau pair parametrized by the stack $\sP_1$ \cite[\S2.1, Notation]{ISZ25}, an enlargement of the GIT moduli stack $\sP_1^{\GIT}$ of pairs $(\bP^1,\tfrac{1}{2}(p_1+\cdots+p_4))$ with GIT semistable $p_1+\cdots+p_4$. 
\item One then runs the relative MMP with scaling for the coarse space of the surface generalized pair  
\[
(\cC^{\tw},\mathbf{M}^{\tw}) \to \Spec A,
\]
and lifts this MMP to a sequence of birational contractions of the coarse space of the threefold pair using \cite[Theorem 1.2]{ISZ25}
\[
\bigl(\mtc{Y}^{\tw},(\tfrac{1}{2}+\epsilon_1)\mtc{R}^{\tw}+\epsilon_2\mtc{F}^{\tw}\bigr) \to \Spec A.
\]
Each step of the surface MMP 
\[
(C^{(i)},\mathbf{M}^{(i)}) \longrightarrow (C^{(i+1)},\mathbf{M}^{(i+1)}),
\]
contracts a rational tail $G$ of $C^{(i)}_0$, and there exists a birational contraction
\[
\bigl(Y^{(i)},(\tfrac{1}{2}+\epsilon_1)R^{(i)}+\epsilon_2F^{(i)}\bigr)
\dashrightarrow
\bigl(Y^{(i+1)},(\tfrac{1}{2}+\epsilon_1)R^{(i+1)}+\epsilon_2F^{(i+1)}\bigr)
\]
such that $Y^{(i+1)}$ is fibered over $C^{(i+1)}$, and $\mathbf{M}^{(i+1)}$ is the moduli part of the canonical bundle formula applied to $(Y^{(i+1)},\tfrac{1}{2}R^{(i+1)}) \to C^{(i+1)}.$ Moreover, $Y^{(i)}\dashrightarrow Y^{(i+1)}$ is an isomorphism over $C^{(i)}\setminus G$. After finitely many steps, one obtains a relative log minimal model $(C^{\st},M^{\st})\rightarrow \Spec A$, and the corresponding threefold pair $(Y^{\st},(\frac{1}{2}+\epsilon_1)R^{\st}+\epsilon_2 F^{\st})\rightarrow \Spec A$ is KSBA-stable.
\item\label{pt_3_ruled_model} Finally one uses that the threefold $Y^{(i)}$ admits a ``ruled model'': a birational model $Y^{(i)}\dashrightarrow Z^{(i)}$ such that $Z^{(i)}$ is easier to describe than $Y^{(i)}$. To obtain our main results we study the birational transformations relating $Z^{(i)}$ and $Y^{(i)}$.
\end{enumerate}

\begin{corollary}\label{cor:fiber over the node}
With the notation as in Theorem~\ref{thm: main theorem of ISZ25}, let $p \in C^{\st}$ be either 
\begin{itemize}
    \item a node, or 
    \item a smooth point of $C^{\st}$, except for finitely many such points.
\end{itemize} Then, locally near $p$, the family $(Y^{\st}, R^{\st}) \longrightarrow (p \in C^{\st})$ is the coarse space of the pullback of the universal family $(\mtc{Y}, \mtc{R}) \longrightarrow (\mtf{p} \in \mtc{C})$ along a stable quasimap from $\mtc{C}$ to $\sP_{1}$. In particular, the fiber 
$(\mtc{Y}^{\st}_{\mtf{p}},\, \mtc{R}^{\st}_{\mtf{p}})$
is such that $\mtc{R}^{\st}_{\mtf{p}}$ is supported over four distinct points.
\end{corollary}

\begin{proof}
By the construction of the birational contractions, the two families of curves $C^{\tw}$ and $C^{\st}$ are isomorphic in a neighborhood of $p$, and likewise the pairs $(Y^{\tw},R^{\tw})$ and $(Y^{\st},R^{\st})$ are isomorphic in a neighborhood of the fiber over $p$. The final assertion follows from the stability conditions for stable quasimaps; see \cite[Definition~2.1]{ISZ25}.
\end{proof}

We conclude this section with the following technical remark, which will be used several times later in the paper. 
Readers who prefer to focus on the main ideas may wish to skip it on a first reading.

\begin{remark}\label{rmk_we_can_ignore_twisted_node}\textup{
Recall that the nodal curves parametrized by $\sP_1$ are \emph{twisted} curves, i.e.\ Deligne--Mumford stacks rather than projective schemes; see \cite[\S2.1]{ISZ25}. 
However, for the purpose of determining KSBA-stable limits, one may replace these twisted curves by their coarse moduli spaces, effectively ignoring the twisted nodes.
More precisely, since all of our results concern the coarse spaces of stacky families of surfaces or curves, we will always take coarse spaces in two steps. 
Given a family $(\cZ,\cD) \to \cC$ fibered in pairs parametrized by $\sP_1$, where $\cC$ is a twisted curve, we first take the relative coarse moduli space
\[
\cZ \longrightarrow \cZ^{\mathrm{rel}}
\]
of $\cZ \to \cC$. 
This replaces each fiber containing a twisted dumbbell curve by the corresponding (untwisted) dumbbell curve, i.e.\ it replaces every twisted curve parametrized by $\sP_1$ with its coarse space. 
We then take the absolute coarse moduli space of $\cZ^{\mathrm{rel}}$; this agrees with taking the coarse space of $\cZ$ directly.
Equivalently, one may describe this procedure using \cite[Section~2.1]{ISZ25}. 
There exists a moduli space $\sP_1^{\CY}$ parametrizing pairs $(C,p_1+\cdots+p_4)$ that arise as the coarse moduli spaces of curves in $\sP_1$. 
Concretely, such pairs satisfy $C \simeq \bP^1$ or $C$ is a nodal union of two copies of $\bP^1$, the pair $(C,\tfrac{1}{2}(p_1+\cdots+p_4))$ is slc, and the line bundle $\omega_C(\tfrac{1}{2}(p_1+\cdots+p_4))$ is trivial. 
There is a natural morphism
\[
\sP_1 \longrightarrow \sP_1^{\CY}
\]
sending each twisted curve to its coarse moduli space.
Thus, given an open substack $\cC$ of a twisted curve and a morphism $\phi\colon \cC \to \sP_1$, let $\cZ \to \cC$ be the associated family of twisted curves. 
To compute the coarse moduli space of $\cZ$, one may instead consider the composition
\[
\cC \longrightarrow \sP_1 \longrightarrow \sP_1^{\CY},
\]
and replace $\cC$ by the relative coarse moduli space $\cC \to \cC'$ of the induced morphism to $\sP_1^{\CY}$. 
The resulting pulled-back family $\cZ' \to \cC'$ has the property that its coarse moduli space agrees with the coarse moduli space of $\cZ$.}

\textup{In summary, for our purposes, we may safely ignore twisted nodes and postcompose all morphisms $\cC \to \sP_1$ with $\sP_1\to \sP_1^{\CY}$. 
While working with $\sP_1$ is convenient for proving \eqref{pt_3_ruled_model}, once this result is established one may work directly with $\sP_1^{\CY}$.}
\end{remark}

\section{Moduli of elliptic surfaces with a bisection}\label{sec:Moduli of elliptic surfaces with a bisection}


Let $\overline{\mathcal{M}}^{\KSBA}_{\mathrm{ell},2}(\epsilon_1,\epsilon_2)$ denote the irreducible closed substack of the KSBA moduli stack that generically parametrizes pairs $(X, \epsilon_1 R + \epsilon_2 F)$ with $0<\epsilon_1 \ll \epsilon_2 \ll 1$, where:
\begin{itemize}
    \item $f: X \to C$ is a smooth elliptic surface over a smooth curve, with no sections, but admitting a (non-unique) bisection $B$;

    \item $R \subseteq X$ is the ramification divisor of the quotient $X \to Y$ induced by the linear system $|B|$;

    \item $f$ has at worst irreducible nodal fibers, and $F$ denotes the sum of all nodal fibers;

    \item the Kodaira dimension satisfies $\kappa(X) \ge 0$, ensuring that the moduli stack is non-empty.
\end{itemize}


The main goal of this section is to establish the following structure theorem for pairs parametrized by the moduli stack, together with a complete classification of the singular fibers and the corresponding surface singularities.

\begin{theorem}\label{thm:elliptic with a bisection moduli}
    Any pair $(X,\epsilon_1 R+\epsilon_2F)$ parametrized by $\ove{\mtc{M}}^{\KSBA}_{\textup{ell},2}(\epsilon_1,\epsilon_2)$ admits a fibration $$f\ :\ (X,\epsilon_1 R+\epsilon_2F)\ \longrightarrow \ C$$ to a nodal curve $C$. Moreover, $\Supp F$ is the union of some singular $f$-fibers, and for any irreducible component $G$ of $C$ with normalization $G^\nu$, the surface $X|_{G^\nu}$ is the coarse space of one of the following:
    \begin{enumerate}
        \item a normal elliptic surface $X_{\mtc{G}}\rightarrow \mtc{G}$; or 
        \item a non-normal fibered surface $X_{\mtc{G}}\rightarrow {\mtc{G}}$ whose general fiber is a banana curve.
    \end{enumerate}
\end{theorem}
When we say elliptic surface, we mean a surface $X$ with a morphism $X\to C$ whose generic fiber is a smooth curve of genus one.

The main idea to prove Theorem \ref{thm:elliptic with a bisection moduli} is to reduce to the results in \cite{ISZ25} as follows. 
\begin{construction}\label{construction}\textup{
Let $f\colon X \to C$ be a normal elliptic fibration over a smooth curve whose singular fibers are all irreducible curves with at worse nodal singularities, and let $B \subseteq X$ be a bisection. Then the linear system $|B|$ is base-point free over $C$, and therefore induces a double cover
\[
\pi: X \longrightarrow 
Y := \Proj_C \bigoplus_{m \ge 0} \Sym^m f_* \mathcal{O}_X(B)
\]
over $C$, where:
\begin{enumerate}
    \item Every $\pi$-fiber is isomorphic to $\mathbb{P}^1$, and fiberwise the map is the standard $2\!:\!1$ cover from an elliptic curve to $\mathbb{P}^1$ determined by a $g^1_2$.
    \item The image of $B$, denoted $S$, is a section of $g : Y \to C$.
    \item The ramification divisor $R \subseteq Y$ (resp.\ $R_X \subseteq X$) of $\pi$ is a $4$-section of $g$ (resp.\ of $f$), and they satisfy the relation $\pi^* R = 2 R_X$.
    \item The pair $(X, 2a R_X + bF)$ is KSBA-stable if and only if the pair $\bigl(Y, (\frac12+a)R + bF\bigr)$ is KSBA-stable, where $0<a,b\le 1$ and $F$ are some marked fibers of $f$ or $g$.
\end{enumerate}}
\end{construction}
The other results of this section regard the singularities of the pairs appearing in $\ove{\mtc{M}}^{\KSBA}_{\textup{ell},2}(\epsilon_1,\epsilon_2)$:

\begin{theorem}[Singularities over the Smooth Locus I]\label{thm:Singularities over the Smooth Locus I}
    With the notation as in Theorem \ref{thm:elliptic with a bisection moduli}, assume that the surface $X_{G^\nu}$ is normal. 
    \begin{itemize}
        \item If a singular fiber of $X_G \to G$ over $G \cap C^{\sm}$ is reduced, then both the fiber and the corresponding singularity of $X_G$ appear in \textup{Table~\ref{table:classification of singularities}}.
        \item If $F_p$ is a non-reduced fiber of $X_G \to G$ over a point $p \in G \cap C^{\sm}$, then $F_p$ has multiplicity~$2$ and its reduced structure satisfies $(F_p)_{\textup{red}} \simeq \bP^1$. In this case, the corresponding singularity of $X_G$ is of type $A_{2k-1}$ (with $k \leq 4$), or of type $D_m$, or of type $E_7$.
    \end{itemize}
\end{theorem}

\begin{theorem}[Singularities over smooth locus II]\label{thm:Singularities over smooth locus II}
 With the notation as above, assume that the surface $X_{G^\nu}$ is non-normal. 
  \begin{itemize}
        \item If a singular fiber of $X_G \to G$ over $G \cap C^{\sm}$ is reduced, then it is either a banana curve, or a $k$-cycle for $k=3,4$.
        \item If $F_p$ is a non-reduced fiber of $X_G \to G$ over a point $p \in G \cap C^{\sm}$, then $F_p$ has multiplicity~$2$ and its reduced structure is a dumbbell curve. In this case, the corresponding singularity of $X_G$ is of type $A_{2k-1}$ (with $k \leq 2$), or of type $D_m$, or of type $E_7$.
    \end{itemize}
\end{theorem}

\begin{theorem}[Singular fibers over nodes]\label{thm:Singular fibers over nodes}
      With the notation as above, a fiber of $X_G\to G$ over $G\cap C^{\sing}$ is one of the following:
\begin{itemize}
\item a reduced curve which is either an elliptic curve or a banana curve, or  
\item a non-reduced curve of multiplicity $m \in \{2,3,4, 6, 8\}$ whose reduced structure is $\bP^1$, or
\item a non-reduced curve of multiplicity $2$ whose reduced structure is an elliptic curve, or
    \item a non-reduced curve whose reduced structure is a banana curve, with multiplicities either $(1,2)$ or $(2,2)$ or $(2,4)$ or $(4,4)$ on its two components.
\end{itemize}
\end{theorem}

\begin{remark}\textup{
The preceding results should be understood as holding only in one direction. 
More precisely, we do not claim that all singularities listed above actually occur on surface pairs parametrized by $\overline{\mtc{M}}^{\KSBA}_{\textup{ell},2}(\epsilon_1,\epsilon_2)$. 
Rather, our claim is that any singularity appearing on a surface pair in $\overline{\mtc{M}}^{\KSBA}_{\textup{ell},2}(\epsilon_1,\epsilon_2)$ must be among those described in the previous results.}
\end{remark}

We now outline the strategy for proving the theorems stated above. We begin by recalling several results concerning KSBA-stable limits of surface pairs equipped with a $\boldsymbol{\mu}_2$-action; see Section~\S\ref{subsec:Cyclic covers of KSBA-stable pairs}. The proof of Theorem~\ref{thm:elliptic with a bisection moduli} follows immediately from the results of \cite{ISZ25}, reviewed in Section~\S\ref{sec:KSBA-moduli stacks for certain ruled surfaces}, together with some preparatory analysis of certain ramified covers of the surface pairs carried out in Section~\S\ref{subsec:Cyclic covers of KSBA-stable pairs}.

The classification of singular fibers requires more work. Theorem~\ref{thm:Singular fibers over nodes}, which treats singular fibers lying over the nodal locus of $C$, is proved in Section~\S\ref{subsec:Classification of fibers over nodal locus of $C$}. The analysis over the smooth locus of $C$ is divided into two cases, according to whether the ramification divisor contains a fiber; these are handled in Sections~\S\ref{subsec:Singular fibers over the smooth locus: no ramified fibers} and~\S\ref{subsec:Singular fibers over the smooth locus: with ramified fibers}, and together, these results complete the proof of Theorem~\ref{thm:Singularities over the Smooth Locus I} and Theorem~\ref{thm:Singularities over smooth locus II}.

\subsection{Cyclic covers of KSBA-stable pairs}\label{subsec:Cyclic covers of KSBA-stable pairs}
We now report two results on cyclic covers which we will use later, many of which are probably known to experts.

\begin{lemma}\label{lem:automorphism extends} 
    Let $({X},{D})\rightarrow \Spec A$ be a normal KSBA-stable family over the spectrum of a DVR $A$. Then any automorphism of the generic fiber $({X}_{\eta},{D}_{\eta})$ extends uniquely to an automorphism of $({X},{D})$ over $\Spec A$.  
\end{lemma}

\begin{proof}This follows as the moduli stack $\overline{\cM}^{\KSBA}$ is separated.
    Indeed, let $x\in \overline{\mtc{M}}^{\KSBA}(R)$ the object corresponding to the family $({X},{D})\rightarrow \Spec A$. As $\overline{\mtc{M}}^{\KSBA}$ is separated, the inertia stack $\Aut(x)\rightarrow \Spec A$ is a proper algebraic space over $\Spec A$, which is a scheme finite over $\Spec A$ as $\overline{\mtc{M}}^{\KSBA}$ is Deligne-Mumford. Therefore, the assertion follows from the valuative criterion of the properness of schemes.
\end{proof}

\begin{corollary}\label{cor_to_take_ksba_limit_one_can_take_quotient_first}
    Let $({X},{D})\rightarrow \Spec A$ be a KSBA-stable family with normal generic fiber over the spectrum of a DVR $A$ and $\tau$ be an automorphism of $( {X}, {D})$ over $\Spec A$ of order $m$. Let \[\begin{tikzcd}[ampersand replacement=\&]
	{ {X}} \&\& { {Y}} \\
	\& {\Spec A}
	\arrow["\pi", from=1-1, to=1-3]
	\arrow[from=1-1, to=2-2]
	\arrow[from=1-3, to=2-2]
\end{tikzcd}\] be the quotient of $ {X}$ under the $\tau$-action, $ {D}_{ {Y}}$ be the a $\bQ$-divisor supported over the image of $ {D}$ under $\pi$ and such that for $U\subseteq Y$ a big open in $Y$ where $D_Y$ is Cartier, we have 
\[
\pi^*(D_Y)|_U = D|_{\pi^{-1}(U)}.
\]
Let finally $ {R}\subseteq  {Y}$ be the divisorial part of the ramification locus. Then $( {Y},\frac{m-1}{m} {R}+ {D}_{ {Y}})\rightarrow \Spec A$ is KSBA-stable. In particular, $ {R}$ does not contain any component of $ {Y}_0$.
\end{corollary}
If $D$ has no component fixed by $\tau$, then one can take as $D_Y$ the image of $D$ by $\tau$.
\begin{proof}
    By the assumptions, we have that $ {Y}$ is normal, $K_{ {Y}}+\frac{m-1}{m} {R}+ {D}_{ {Y}}$ is $\bQ$-Cartier by \cite[Lemma 2.3]{alexeev2012non}, and $$\textstyle K_{ {X}}+ {D}\ =\ \pi^*(K_{ {Y}}+\frac{m-1}{m} {R}+ {D}_{ {Y}}),  \ \ \ \  {X}_0=\pi^* {Y}_0,$$ As $( {X}, {D})\rightarrow \Spec A$ is KSBA-stable, then by the inversion of adjunction, $( {X}, {D}+ {X}_0)$ is log canonical. By \cite[Proposition 5.20]{KM98}, it follows that $( {Y},\frac{m-1}{m} {R}+ {D}_{ {Y}}+ {Y}_0)$ is a log canonical pair whose log canonical divisor is ample over $\Spec A$. Therefore, components of $ {Y}_0$ are not contained in $\Supp  {R}$, and $( {Y},\frac{m-1}{m} {R}+ {D}_{ {Y}})\rightarrow \Spec A$ is KSBA-stable.
\end{proof}

We now record several results specific to our setting, which will be used later in the paper.

\smallskip
Let $(X_0,\epsilon_1R_0+\epsilon_2 F_0)$ be a KSBA-stable pair in $\overline{\mtc{M}}^{\KSBA}_{\textup{ell},2}(\epsilon_1,\epsilon_2)$. Take a DVR $A$ and take a general smoothing $( {X},\epsilon_1 {R}_{ {X}}+\epsilon_2 {F}_{ {X}})\rightarrow \Spec A$ of $(X_0,\epsilon_1R_0+\epsilon_2 F_0)$. Let $( {Y},(\frac{1}{2}+\epsilon_1) {R}+\epsilon_2 {F})\rightarrow \Spec A$ be the quotient induced by a bisection of the generic fiber $ {X}_{\eta}$. Then by Theorem \ref{thm: main theorem of ISZ25} the threefold $ {Y}$ admits a morphism with purely 1-dimensional fibers to a family of at worst nodal curves $ {C}\rightarrow \Spec A$, whose generic fiber $ {C}_{\eta}$ is smooth.

\begin{lemma}\label{lem: pure branch locus 1}
    With the same notations as before, there exists a big open subscheme $C^\circ \subseteq {C}$ such that the ramification locus of ${X} \rightarrow {Y}$ over ${C}^\circ$ is divisorial, i.e.\ every irreducible component has codimension $1$, and is contained in $({Y}|_{{C}^\circ})^{\mathrm{sm}}$. In particular, there exists a line bundle $\mtc{L}^\circ$ on $ {Y}|_{ {C}^\circ}$ such that $ {R}|_{ {C}^\circ}$ is a section of $(\mtc{L}^{\circ})^{\otimes2}$.
\end{lemma}
\begin{proof}
    Since ${C} \to \Spec A$ is a family of nodal curves with smooth geometric generic fiber, it suffices to work over the smooth locus $ {C}^{\mathrm{sm}}$ of $ {C} \to \Spec A$. 
    
    Let $p \in  {C}^{\mathrm{sm}}$ be a point. If the fiber of $ {Y} \to  {C}$ over $p$ is $\mathbb{P}^1$, then in a neighborhood $U$ of $p$ the morphism $ {Y}|_U \to U$ is a $\mathbb{P}^1$–fibration. By purity of the branch locus \cite[Tag~0BMB]{stacks-project}, the ramification locus over $U$ is divisorial and lies in $( {Y}|_U)^{\mathrm{sm}}$.

    Now assume the fiber of $ {Y} \to  {C}$ over $p$ is a dumbbell curve, defined in Definition \ref{def_dumbell_curve}. By Theorem~\ref{thm: main theorem of ISZ25}, there exists an irreducible component $G \subseteq  {C}_0^{\mathrm{sm}}$
    such that $ {Y}|_G \to G$ is a fibration whose general fiber is a dumbbell curve. We claim that the node of a general fiber of $ {Y}|_G \to G$ is not contained in the branch locus. 

    Indeed, a double cover of a dumbbell curve $\Gamma$ that is branched at the node together with four additional points (two on each component) is never nodal; its normalization acquires worse-than-nodal singularities at the preimage of the node. Since by assumption $ {X}|_G \to G$ has generically at worst nodal singularities, the double cover cannot branch at the node. Thus the branch locus avoids the nodes of fibers, and hence is contained in the smooth locus of $ {Y}|_G$.

    Combining the two cases gives the desired big open subset $ {C}^\circ$.
\end{proof}
The following is a lemma which will be useful for the remaining part of the paper.
\begin{lemma}\label{lem:isom in codim 1 implies isom}
Let $X$ and $X'$ be $S_2$ varieties, and let $f \colon X \to C$ and $f' \colon X' \to C$ be projective morphisms to a base scheme $C$ with pure-dimensional fibers. Assume that there is a birational map 
\[
\pi \colon X \dashrightarrow X'
\]
over $C$ which is an isomorphism over $C^{\circ} := C \setminus Z$, where $Z$ is a closed subscheme of codimension $\geq2$. Suppose further that $L$ and $L'$ are $\bQ$-Cartier $\bQ$-divisors on $X$ and $X'$, respectively, which are relatively ample over $C$, and that 
\[
L|_{f^{-1}(C^{\circ})} = L'|_{f'^{-1}(C^{\circ})}
\]
under the identification induced by~$\pi$. Then $X$ and $X'$ are isomorphic over $C$, and under this isomorphism $L$ corresponds to $L'$.
\end{lemma}

\begin{proof}
Choose an integer $r>0$ such that $rL$ and $rL'$ are Cartier. For every integer $m\ge 0$ set
\[
\mathcal{A}_m := f_*\mathcal{O}_X(mrL), \qquad 
\mathcal{A}'_m := f'_*\mathcal{O}_{X'}(mrL').
\]
Since $mrL$ and $mrL'$ are relatively ample over $C$, the graded $\mathcal{O}_C$-algebras
\[
\mathcal{A} := \bigoplus_{m\ge 0} \mathcal{A}_m, 
\qquad
\mathcal{A}' := \bigoplus_{m\ge 0} \mathcal{A}'_m
\]
are finitely generated, and we have isomorphisms
\[
X \simeq \Proj_C(\mathcal{A}), \qquad 
X' \simeq \Proj_C(\mathcal{A}'),
\]
with $\mathcal{O}_{\Proj}(1)$ corresponding to $rL$ and $rL'$, respectively. Over $C^\circ$, the birational map $\pi$ is an isomorphism, and by assumption
\[
L|_{f^{-1}(C^\circ)} \simeq L'|_{f'^{-1}(C^\circ)}.
\]
Hence for every $m\ge 0$ we obtain an isomorphism of $\mathcal{O}_{C^\circ}$-modules
\[
\mathcal{A}_m|_{C^\circ} \;\simeq\; \mathcal{A}'_m|_{C^\circ}.
\] We may assume that $C$ is affine. As $\codim_XZ\geq 2$ and $f$ has pure dimensional fibers, and as $$\Gamma(\mtc{A}_m,C)= H^0(X,mr{L})= H^0(X|_{C^\circ},mr{L}|_{C^\circ})=H^0(X'|_{C^\circ},mr{L}'|_{C^\circ})=\Gamma(\mtc{A}'_m,C),$$ the isomorphism over $C^\circ$ extends uniquely to an isomorphism $\mathcal{A}_m \xrightarrow{\ \simeq\ } \mathcal{A}'_m.$ These isomorphisms are compatible with the natural multiplication maps, and therefore assemble to an isomorphism of graded $\mathcal{O}_C$-algebras $
\mathcal{A} \;\simeq \; \mathcal{A}'$. Taking relative $\Proj$ over $C$ then gives the desired isomorphism $X\simeq X$, and under this isomorphism, the class of $rL$ corresponds to the class of $rL'$. Hence $L$ corresponds to $L'$ as $\mathbb{Q}$-Cartier $\mathbb{Q}$-divisors.
\end{proof}

\begin{corollary}\label{cor:pure branch locus 2}
    In Lemma~\ref{lem: pure branch locus 1}, one can take $ {C}^\circ =  {C}^{\mathrm{sm}}$.
\end{corollary}

\begin{proof}
    This follows from Lemma~\ref{lem:isom in codim 1 implies isom}.  
    Indeed, consider the double cover of $ {Y}|_{ {C}^{\mathrm{sm}}}$ branched along 
    $ {R}|_{ {C}^{\mathrm{sm}}}$.  
    By construction, the inverse image of $ {R}|_{ {C}^{\mathrm{sm}}}$ is relatively ample over 
    $ {C}^{\mathrm{sm}}$. Hence the resulting double cover satisfies the hypotheses of 
    Lemma~\ref{lem:isom in codim 1 implies isom}, and must therefore be isomorphic to 
    $ {X}$ over $ {C}^{\mathrm{sm}}$.  
    This shows that no further shrinking of the base is necessary, so one may take 
    $ {C}^\circ =  {C}^{\mathrm{sm}}$.
\end{proof}

\subsection{Classification of fibers on double covers: preparation}\label{subsec:Classification of fibers on double covers: preparation}

One of the main goals of this section is to classify the singular fibers that appear by taking the double cover of a surface pair $(Y,\frac{1}{2}R)$ such that $(Y,(\frac{1}{2}+\epsilon_1)R + \epsilon_2F)$ is a surface pair in $\overline{\cM}^{\KSBA}(\epsilon_1,\epsilon_2)$. As preparation, we first classify certain ramified double covers of curves.

\begin{defn}\label{def_dumbell_curve}
A \emph{dumbbell curve} is the nodal union of two smooth rational curves meeting at a single point. A \emph{twisted dumbbell curve} is a twisted curve $\mtc{C}$ whose coarse space is a dumbbell curve $C$, and which has a unique stacky point lying over the node of $C$ with stabilizer $\mu_2$. The stabilizer preserves both components and acts on each component via $z \mapsto z^2$. 
\end{defn}

\noindent Equivalently, a twisted dumbbell curve is obtained by taking the nodal union of the two $\mu_2$–root stacks of $\bP^1$ at $\infty$, and gluing them along their common closed substack $\mtc{B}\mu_2$.

\begin{defn}
    A reduced projective curve $C$ of arithmetic genus $1$ is called 
    \begin{itemize}
        \item a \emph{banana curve} if $C=C_1\cup C_2$ is the nodal union at two points of two smooth rational curves;
        \item a \emph{pinched banana curve} if $C=C_1\cup C_2$ is the union of two smooth rational curves at a singular point, which is a tacnode of $C$;
        \item an \emph{$n$-cycle}, for $n\geq 3$, if $C=C_1\cup \cdots\cup C_n$ is the nodal union of $n$ smooth rational curves such that $$C_i\cap C_j\ =\ \begin{cases}
            \textup{singleton}, & |i-j|=1\\
            \emptyset, & |i-j|\geq 2
        \end{cases}$$ for any $i,j=1,...,n$ where $C_{n+1}:=C_2$ and $C_{n+2}:=C_{2}$.
       \item A \emph{four-leaf-clover curve} if \(C = C_{1} \cup_{p} C_{2} \cup_{p} C_{3} \cup_{p} C_{4}\) is the union of four smooth rational curves meeting at a common point \(p\), each pair of components intersecting transversely, and such that the embedded dimension of \(C\) at \(p\) is \(3\). Equivalently, \(C\) is isomorphic to \(\bV(xy,\, z(x+y+z)) \subseteq \bP^{3}\).
    \end{itemize}
\end{defn}

The next two lemmas will be used later in the classification of the singular fibers, whose proofs involve only simple computations.

\begin{lemma}
    The double cover of $\bP^1$ branched along a divisor $D$ of degree $4$ is one of the following:
    \begin{itemize}
        \item a smooth elliptic curve, if $\Supp D$ consists of four distinct points;
        \item an irreducible nodal genus one curve, if $\Supp D$ consists of three distinct points;
        \item an irreducible cuspidal genus one curve, if $\Supp D$ consists of two distinct points, one of which has multiplicity $3$;
        \item a banana curve, if $\Supp D$ consists of two distinct points, each of which has multiplicity $2$;
        \item a pinched banana curve, if $\Supp D$ consists of a single point.
    \end{itemize}
\end{lemma}

\begin{lemma}
    Let $C=C_1\cup C_2$ be a dumbbell curve, and $D\subseteq C^{\sm}$ be a divisor on $C$ such that $\deg D|_{C_i}=2$ for $i=1,2$. Then the double cover of $\bP^1$ branched along a divisor $D$ of degree $4$ is one of the following:
    \begin{itemize}
    \item a banana curve, if $\Supp D$ consists of four distinct points;
    \item a 3-cycle, if $\Supp D$ consists of three distinct points;
    \item a 4-cycle, if $\Supp D$ consists of two distinct points;
    \end{itemize}
\end{lemma}

\begin{lemma}\label{lem: special double cover of dumbbell 1}
Let $C = C_1 \cup_{p} C_2$ be a dumbbell curve with the node $p$. Let $D \subset C$ be a Cartier divisor such that for $i = 1,2$ one has $\deg(D|_{C_i}) = 2$ and $\Supp(D|_{C_i}) = \{p, p_i\}$, where $p_i\in C_i\setminus\{p\}$. Then the double cover of $\bP^1$ branched along $D$ is a pinched banana curve.
\end{lemma}

\begin{proof}
    One may verify the claim by writing a local equation near the node $p$.  
    Locally at $p$, the dumbbell curve $C$ is given in $\bA^2$ by the equation $xy = 0$, and the divisor $D$ is \'{e}tale locally the restriction of the line $\bV(x-y)$. Hence the double cover of $C$ branched along $D$ is locally given by
    \[
        \Spec k[x,y,w]/(w^2 - (x-y),\, xy) \ = \   \Spec k[y,w]/\bigl(y(y+w^2)\bigr),
    \]
    which is the standard local equation of a tacnode. Away from $p$, the cover is \'{e}tale over each component $C_i$ except at the two branch points $\{p,q_i\}$; thus its normalization restricts to a double cover of $\bP^1$ branched over two distinct points, which is again isomorphic to $\bP^1$.  Therefore, the resulting curve is a pinched banana curve.
\end{proof}

\begin{lemma}\label{lem: special double cover of dumbbell 2}
    Let $C = C_{1} \cup_{p} C_{2}$ be a dumbbell curve. 
    Let $D \subset C$ be a Cartier divisor such that for $i = 1,2$ one has 
    $\deg(D|_{C_{i}})=2$ and $\Supp(D|_{C_{i}})=\{p\}$. 
    Then the double cover of $C$ branched along $D$ is a four-leaf-clover curve.
\end{lemma}

\begin{proof}
    The same argument as in the proof of Lemma~\ref{lem: special double cover of dumbbell 1} applies.  
    Locally at $p$, the curve $C$ is given by the node $xy=0$, and the divisor $D$ is the
    restriction of the non-reduced line $\bV((x+y)^{2})$.  
    Therefore, the local equation of the double cover is
    \[
        \Spec k[x,y,z]/(z^{2}-(x+y)^{2},\,xy)
        \;=\;
        \Spec k[x,y,z]/\bigl((z-(x+y))(z+(x+y)),\,xy\bigr).
    \]
    This exhibits four smooth branches meeting transversely at the origin in 
    embedding dimension~$3$, hence the double cover is a four-leaf-clover curve.
\end{proof}

\subsection{Classification of fibers over nodal locus of $C$}\label{subsec:Classification of fibers over nodal locus of $C$}

Let $(Y,(\frac{1}{2}+\epsilon_1)R+\epsilon_2F)\in \overline{\mtc{M}}^{\KSBA}(\epsilon_1,\epsilon_2)$ be a stable pair, and assume it is the central fiber of a one-parameter family over the spectrum of a DVR $A$ \[\textstyle \left(\cY,\left(\frac{1}{2}+\epsilon_1\right)\cR+\epsilon_2\cF\right) \ \longrightarrow\  \cC \longrightarrow \   \spec A\] fibered over a twisted curve $\cC$, whose generic fiber is a surface ruled over a smooth curve $\cC_\eta$. To get the corresponding family of elliptic surfaces, one takes the double cover ramified over $\cR$, which we denote by $\cX\to \cY$.
In this subsection, we consider a node $\frak{p}\in \cC_0$ (resp. $p\in C_0$) in the central fiber of $\cC\to \spec A$ (resp $C\to \spec A$, where $\cC\to C$ is the coarse space map), and we study the fibers $\cX_p$ and $\cX_{\frak{p}}$.
By Lemma \ref{lem: pure branch locus 1}, there exists a $\bQ$-Cartier Weil divisor $\mtc{L}$ on $\mtc{Y}$ with Cartier index $\leq 2$ such that $\mtc{R}$ is a section of $\mtc{L}^{\otimes 2}$. Observe that $\mtc{R}$ is a Cartier divisor on $\cY$, so there are two possibilities: either $\cL$ is Cartier over $p$, or it is not Cartier but $\cL^{\otimes 2}$ is Cartier. We will study these two cases separately.

\begin{notation}\textup{
Let $\mtc{G}_1$ and $\mtc{G}_2$ be the two irreducible components of $\cC_0$  passing through $p$ in the \'{e}tale topology\footnote{Here $p$ could be a node of a single irreducible component in the Zariski topology of $C$, so \'etale locally the node has two branches}. Similarly, let $\cY_1$ and $\cY_2$ denote the surfaces over $\cG_1$ and $\cG_2$, and similarly with $\cX_i$. Denote by $G_1,G_2,X_1,X_2,Y_1,Y_2$ the corresponding coarse spaces. The fiber of $\cY_i$ over $\mtf{p}$, denoted by $\mtc{Y}_\mtf{p}$, is either $\bP^1$ or a dumbbell curve by Corollary \ref{cor:fiber over the node} and Remark \ref{rmk_we_can_ignore_twisted_node}.
}\end{notation}


\subsubsection{$\cL$ is Cartier}
In this case, we have line bundles $\cL_i$ on $\cY_i$ with a section $\cR_i:=\cR|_{\cY_i}$ of $\cL_i^{\otimes 2}$. So we can take the double cover of $\cY_i$ ramified over $\cR_i$, this agrees with $\cX_i$. 

\begin{prop}
    With the previous notations, we have the following.
    \begin{enumerate}
        \item If $\cY_\frak{p}\simeq \bP^1$, then the fiber $(X_i)_p$ has multiplicity $|\stab_{\cC}(\frak{p})|$. The reduced structure of the fiber is a quotient of an elliptic curve by a (possibly trivial) cyclic group.
        \item If $\cY_{\frak p}$ is a twisted dumbbell curve, then the fiber $(X_i)_{\frak p}$ has either
\begin{itemize}
  \item an irreducible component of multiplicity $2$ or $4$ with reduced structure $\bP^1$, or
  \item two irreducible components of multiplicities $(1,2)$ or $(2,2)$, whose reduced structure is a dumbbell curve.
\end{itemize}

    \end{enumerate}
    
\end{prop}

\begin{proof} Let $\bmu_n=\stab_{\cC}(\frak{p})$.
    We begin with the case when $\mtc{Y}_{\mtf{p}}\simeq \bP^1$. As $\cL$ is Cartier, the map $\cX\to \cY$ is generically \'etale over $\cY_p$, and thus the multiplicities of the fibers $(X_i)_p$ and $(Y_i)_p$ agree. Indeed, we can write $\cX$ as \[
    \spec_{\cY}(\cO_{\cY}\oplus \cL^{-1})
    \] with the usual algebra structure of $\bmu_2$-covers (see e.g. \cite[Definition 2.52]{KM98}), and check explicitly this is generically \'etale in codimension one, away from the ramification locus $\cR$. We now compute the multiplicity of $(Y_i)_p$. To ease the notations, we drop the subscript $i$ as it plays no role.

    We first show that $\cY\times_\cG \cB\bmu_n$ is generically a scheme, where the map $\cB\bmu_n\to \cG$ is the closed embedding of the residual gerbe at $\frak{p}$. Indeed, locally near $\mtf{p}$, $\cY$ is the pull-back of the universal family via a \textit{representable} morphism $\phi\colon \cG\to \sP_{1}$, hence $\cY\times_\cG \cB\bmu_n$ can be identified with the stack-quotient of the curve corresponding to $\phi(\frak{p})$ in $\sP_{1}$ by the action of $\bmu_n$ via the morphism \[\stab_{\cG}(\frak{p})\to \stab_{\sP_{1}}(\phi(\frak{p})).\]
    The latter is injective as $\phi$ is representable, so $\cY\times_\cG \cB\bmu_n$ is a stack-quotient of the form $[\bP^1/\bmu_n]$ with the action of $\bmu_n$ being faithful. In particular, it is generically a scheme. Then the multiplicities of $\cY_p$ and $Y_p$ agree, and we can compute the former. We have
    \[
    \cY_p = \cY\times_{G}\{p\} = \cY\times_{\cG}(\cG\times_G\{p\})
    \]
and 
    \[
    \cG\times_G\{p\} = [\spec (k[\![x]\!]/x^n)/\bmu_n]
    \]
    where $\bmu_n$ acts by $\zeta *x = \zeta x$, as the map $\cG\to G$ is, locally around $p$, a root stack at $p$. Then 
    \[
    \cY_p = \cY\times_{\cG}[\spec (k[\![x]\!]/x^n)/\bmu_n].
    \]
    Now, $\cY_p$ is generically a scheme, as its reduced structure (namely, $\cY\times_\cG\cB \bmu_n\simeq [\bP^1/\bmu_n]$) is generically a scheme. Let $\spec A\to \cY_p$ be an \'etale neighborhood of a schematic point. As the multiplicity of the fiber can be computed \'etale locally, and as the reduced structure of $\cY_p$ agrees with $\cY\times_\cG \cB \bmu_n$, the multiplicity of the fiber agrees with the multiplicity of $\spec A_{\textup{red}}$ in $\spec A$. Consider the following fiber diagrams:
\[
\xymatrix{
(\Spec A)\times \bmu_n \ar[r] \ar[d] 
& \Spec A \ar[d] \\
\Spec k[\![x]\!]/(x^n) \ar[r] 
& [\Spec\bigl(k[\![x]\!]/(x^n)\bigr)/\bmu_n]
}
\qquad
\xymatrix{
(\Spec A_{\mathrm{red}})\times \bmu_n \ar[r] \ar[d] 
& (\Spec A)\times \bmu_n \ar[d] \\
\Spec k \ar[r] 
& \Spec k[\![x]\!]/(x^n).
}
\]
Here we shrink the \'etale neighborhood $\Spec A$ further so that the $\bmu_n$-torsor on the left becomes trivial. On the right, we again use the fact that $\cY_{\frak p}\simeq \bP^1$, which implies that the fiber product is reduced. Since the morphism $\cY \to \cG$ is flat, then $(\Spec A)\times \bmu_n \to \Spec k[\![x]\!]/(x^n) $ is flat as well. Consequently, the fiber has multiplicity $n$.

    Regarding the second statement, observe that $(X_i)_{p,\textup{red}}$ is the coarse moduli space of the double cover of $(\cY\times_\cG \cB \bmu_n,\frac{1}{2}\cR\times_\cG \cB \bmu_n)$. In other terms, it is the coarse moduli space $E/\!\!/\bmu_n$ of $[E/\bmu_n]$ where $E$ is the elliptic curve that is the double cover $(\bP^1,\cR\times_\cG\frak{p})$.

    We now address the second point. Proceeding as before, we have that $\cY\times_\cG\cB \bmu_n = [C/\bmu_n]$ where $C$ is the dumbbell curve\footnote{the dumbbell curve rather than the twisted dumbbell curve from Remark \ref{rmk_we_can_ignore_twisted_node}}, and $\bmu_n$ is a subgroup of $\Aut(C,R)$ where $R$ consists of two distinct points on each branch of $C$. If $\bmu_n$ does not fix an irreducible component of $C$, one can proceed exactly as before. If instead $\bmu_n$ fixes an irreducible component of $C$ and acts nontrivially on the other one, then $n=2$. In this case, the fiber $\cY_p$ will have an irreducible component of multiplicity 2 (corresponding to the componet where $\bmu_2$ acts nontrivially) and an irreducible component of multiplicity 1. 
\end{proof}

\subsubsection{$\cL$ is of Cartier index~$2$}
In this case, \'etale locally at the generic points $\xi$ of $\cY_{\frak p}$, the morphism $\cY \to \sC$ is isomorphic to
\[
\Spec k[t,x,y,\pi]/(xy-\pi^{2n})
\;\longrightarrow\;
\Spec k[x,y,\pi]/(xy-\pi^{2n}).
\]
In particular, the local fundamental group at $\xi$ is isomorphic to $\bZ/2n\bZ$, and hence, topologically, the only nontrivial double cover is
\[
\Spec k[t,x',y',\pi]/(x'y'-\pi^{n}) 
\;\longrightarrow\;
\Spec k[t,x,y,\pi]/(xy-\pi^{2n}) ,
\]
where $x \mapsto (x')^2$ and $y \mapsto (y')^2$. The set-theoretic fiber of the ramified cover $\cX \to \cY$ consists of a \emph{single} point over every point of $\cY_{\frak p}$. 
The same holds for the induced morphism $\cX_i \to \cY_i$: the map $\cX_i \to \cY_i$ is again a $\bmu_2$-cover, and its codimension-one ramification locus can be identified as the set of codimension-one points of $\cY_i$ over which the fiber of $\cX_i \to \cY_i$ consists of a single set-theoretic point. 
In particular, in our situation, the morphism $\cX_i \to \cY_i$ is ramified along both $\cR|_{\cY_i}$ and $\cY_{\frak p}$. 
Ramification along $\cY_{\frak p}$ has the effect of preserving the reduced structure of the double locus of the central fiber of $\cX$, while doubling its multiplicity.

 \begin{theorem}[Fibers over nodes]\label{thm:classification of fibers over nodes}
     With the notation as above, a fiber of $X_G\to G$ over $G\cap C^{\sing}$ is one of the following:
\begin{itemize}
\item a reduced curve which is either an elliptic curve or a banana curve, or  \item a non-reduced curve of multiplicity $m \in \{2,3,4, 6, 8\}$ whose reduced structure is $\bP^1$, or
    \item a non-reduced curve whose reduced structure is a banana curve, with multiplicities either $(1,2)$ or $(2,2)$ or $(2,4)$ or $(4,4)$ on its two components.
   \item a non-reduced curve of multiplicity $2$ with reduced structure an elliptic curve.
\end{itemize}
\end{theorem}
The first bullet point is the case when the twisted curve $\cC$ is a scheme at the nodes of the special fiber, i.e. with the previous notation the case in which $\bmu_n=\{1\}$. The second case is when the fiber of the morphism $\cY\to \cC$ at a node of the special fiber of $\cC$ is $\bP^1$, or a dumbbell curve $\sY_{\frak{p}}$ with $\bmu_n$ acting transitively on the irreducible components of $\sY_{\frak{p}}$ (so $n=4$). The third case is when $\sY_{\frak{p}}$ is a dumbbell curve and $\bmu_n$ does not act transitively on its irreducible components.

\subsection{Singular fibers over the smooth locus: no ramified fibers}\label{subsec:Singular fibers over the smooth locus: no ramified fibers}

Let \((Y,(\tfrac{1}{2}+\epsilon_1)R+\epsilon_2 F)\) be a stable pair parametrized by 
\(\overline{\mathcal{M}}^{\KSBA}(\epsilon_1,\epsilon_2)\).  
By Theorem~\ref{thm: main theorem of ISZ25}, such a pair admits a fibration 
\(Y \to C\) onto a nodal curve \(C\).  
In this subsection, we classify the singularities of the components of the double cover of \(Y\) branched along \(R\), restricted over the smooth locus \(C^{\sm}\). By Remark \ref{rem:no difference over smooth locus}, it suffices to consider surfaces of either type I or type III.

Let $G$ be an irreducible component of $C$, and denote by 
$C^{\circ} = G \cap C^{\sm}$ the smooth locus of $C$ along $G$.

\begin{prop}
   Assume that the surface
    \(Y_G \to G\) is of type~\textup{I}.  
    Let \(X_G\) be the double cover of \(Y_G\) ramified along \(R|_G\).  
    Then over $G^{\circ}$, all possible singularities of \(X_G\) are exactly those listed in Table~\ref{table:classification of singularities}.
\end{prop}

\begin{proof}
    Since \(Y_G\) is of type~\textup{I}, Corollary~\ref{cor:projective bundle over smooth locus} implies that 
    \(Y_{G^\circ} \to G^\circ\) is a $\bP^1$-bundle.  
    The non-slc fibers of the pair \((Y_{G^\circ}, \tfrac12 R|_{G^\circ})\) were classified in 
    \cite[§4.2.1]{ISZ25}.  
    Taking the double cover branched along \(R|_{G^\circ}\), one immediately obtains the corresponding singularities on \(X_G\).

    We illustrate the argument with one example; all other cases follow by the same local analysis.  
    Suppose that \(R\) has at a point \(p\) a singularity locally given by $y^2 - z^{k+1} = 0$, where \(z=0\) cuts out the fiber of \(Y_G \to G\) through \(p\).  
    Then the ramified double cover has local equation
    \[
        x^2 + y^2 + z^{k+1} = 0,
    \]
    which is an \(A_k\)-singularity.  
    By running through all singularities appearing in \cite[Tables~2--3]{ISZ25}, one obtains precisely the list of singularities of \(X_{G^\circ}\) recorded in Table~\ref{table:classification of singularities}.
\end{proof}

\renewcommand{\arraystretch}{1.5}
\begin{longtable}{| p{.18\textwidth} | p{.15\textwidth} | p{.30\textwidth} | }

\hline
\textbf{Fiber } & \textbf{Surface } & \textbf{Local equation}  \\ \hline
\endfirsthead
\endhead
$A_k$, $k=1,2,3$ & smooth &  \\ \hline
$A_1$ & $A_k$, \ $k\geq 1 $ &  $x^2+y^2+z^{k+1}=0$  \\ \hline
$A_k$, $k=2,3$ & $A_1$ &  $x^2+yz+y^{k+1}=0$  \\ \hline
$A_k$, $k=2,3$ & $A_2$ &  $x^2+z^2+y^{k+1}=0$  \\ \hline
$A_2$ & $D_n$, $n\geq 4$ &  $x^2+(y+z)(y^2-z^{n-2})=0$ \\ \hline
$A_3$ & $D_n$, $n\geq 4$ &  $x^2+(y^2-z^{n-2})(z-y^2)=0$ \\ \hline
$A_2$ & $E_6$ &  $x^2+y^3+z^4=0$  \\ \hline
$A_2$ & $E_8$ &  $x^2+y^3+z^5=0$  \\ \hline
$A_3$ & $E_6$ &  $x^2+z^3+y^4=0$  \\ \hline
$A_2$ & $E_{6k+1}$,\  $k\geq 1$ &  $x^2+y^3+yz^{2k+1}=0$  \\ \hline
$A_2$ & $J_{k,0}$, \ $k\geq 2$ &  $x^2+y^3+yz^{2k}=0$  \\ \hline
\caption{Classification of singularities of fibers and surfaces\\ The equation of the fiber is $z=0$} 
\label{table:classification of singularities}
\end{longtable}

\begin{prop}
    Assume that the surface \(Y_G \to G\) is of type~\textup{III}, i.e.\ it is 
    generically nodal along a curve \(\Gamma\), with normalization $Y^{\nu}_G$.  
    If \(R\) intersects \(\Gamma\) at a point \(q\), then on each component of $Y^{\nu}_G$, the restriction of \(R\) intersects \(\Gamma\). Moreover, on any component of $Y^{\nu}_G$, $\Gamma$ and $R$ intersect transversely. In particular, the double cover \(X_G\) of \(Y_G\) branched along \(R_G\) 
    has a fiber which is either a pinched banana curve or a four-leaf-clover curve.
\end{prop}

\begin{proof}
    For the first statement, since \(R\) is a \(\bQ\)-Cartier divisor, if it meets 
    the double locus \(\Gamma\), then its pullback to the normalization of 
    \(Y_G\) must meet the preimage of~\(\Gamma\) on every irreducible component.  
    This shows that the restriction of \(R\) intersects each branch of \(\Gamma\). If $R$ and $\Gamma$ do not meet transversely at $q$, then the pair $\bigl(Y^{\nu}_G,\; (\tfrac{1}{2}+\epsilon)R|_{Y^{\nu}_G} + \Gamma \bigr)$
is not log canonical at $q$. The last statement follows directly from 
    Lemma~\ref{lem: special double cover of dumbbell 1} and 
    Lemma~\ref{lem: special double cover of dumbbell 2}.
\end{proof}

\begin{prop}
    Assume that the morphism \(Y_G \to G\) is of type~\textup{III}, i.e.\ it is 
    generically nodal along a curve \(\Gamma \subset Y_G\).  
    Then away from \(\Gamma\), the double cover \(X_G\) of \(Y_G\) branched along 
    \(R|_G\) has at worst \(A_k\)-singularities over $G^{\circ}$.
\end{prop}

\begin{proof}
    On each component of the normalization of \(Y_G\), the divisor \(R_G := R|_G\) 
    is a bisection of the projection to \(G\).  
    Over $G^{\circ}$, the morphism \(Y_G \to G\) is smooth, so the only 
    possible singularities of \(R_G\) are ordinary double points.  
    Analytically, these are locally given by an equation of the form
    \(
        x^2 - y^{k+1} = 0 .
    \)
    The double cover of a smooth surface branched along such a hypersurface 
    has an \(A_k\)-singularity.  
    Therefore, away from~\(\Gamma\), the double cover \(X_G\) has at worst 
    \(A_k\)-singularities over $G^{\circ}$.
\end{proof}

\subsection{Singular fibers over the smooth locus: with ramified fibers}\label{subsec:Singular fibers over the smooth locus: with ramified fibers}

In this subsection we consider the case when the ramification divisor 
\(R\) contains a fiber \(F_p\) of the morphism \(Y \to C\) over a point 
\(p \in C\). Here, \((Y,(\tfrac{1}{2}+\epsilon_1)R+\epsilon_2 F)\) is a stable pair parametrized by 
\(\overline{\mathcal{M}}^{\KSBA}(\epsilon_1,\epsilon_2)\), which is necessarily fibered over $C$.

\begin{lemma}\label{lem:F_not_double_locus}
    The point \(p\) is a smooth point of \(C\); in particular, the fiber 
    \(F_p\) cannot lie in the double locus of \(Y\).  
    Moreover, the multiplicity of \(R\) along \(F_p\) is equal to \(1\).
\end{lemma}

\begin{proof}
    If \(p\) were a singular point of \(C\) or if the multiplicity of \(R\) along 
    \(F\) were greater than \(1\), then the pair 
    \((Y, (\tfrac{1}{2}+\epsilon)R)\) would fail to be \emph{slc}, 
    contradicting our assumptions.
\end{proof}

Now let $G$ be the component of $C$ such that $p \in G$, let 
$G^\circ := G \cap C^{\sm}$, and let $Y_G$ be the component of $Y$ lying over $G$.  
Let $X_G$ be the component of the double cover of $Y$ branched along $R$ 
that lies over~$G$.

\begin{lemma}\label{lem:at worst double point}
    Assume that $Y_G$ is of type~\textup{I} or~\textup{II}.  
    Let $q \in F_p$ be a singular point of $R$, and set 
    $\Theta := R|_G - F_p$ to be the residual curve to $F_p$ in $R|_G$.  
    Then:
    \begin{itemize}
        \item $\Theta$ has at worst a double point at $q$; and
        \item if $\Theta$ is smooth at $q$, then the local intersection number 
        $(\Theta, F_p)_q 
        \le 4$. 
    \end{itemize}
\end{lemma}

\begin{proof}
    If $\mult_q(\Theta) \ge 3$, then the pair 
    $(Y,(\tfrac{1}{2}+\epsilon)R)$ would have worse than slc singularities 
    at~$q$, which contradicts our assumptions.  
    The second assertion follows from the fact that generically 
    $\Theta$ is a $4$-section of the fibration $Y_G \to G$.
\end{proof}

\begin{prop}\label{prop:singularity of type I and II}
    Assume that $Y_G$ is of type~\textup{I} or~\textup{II}.  
    Then $X_G \to G$ has a double fiber over $p$, whose reduced structure is 
    $\bP^1$.  
    Locally over $p \in G^\circ$:
    \begin{enumerate}
        \item If $\Theta$ is smooth at $q$, then $X_G$ has an 
        $A_{2k-1}$-singularity with $k \le 4$.
        \item If $\Theta$ is singular at $q$, then $X_G$ has either a 
        $D_k$-singularity with $k \ge 4$, or an $E_7$-singularity.
    \end{enumerate}
\end{prop}

\begin{proof}
    For the first statement, write local equations as follows.  
    Let $x=0$ be a local equation of $F_p$.  
    The double cover of $Y$ branched along $R$ has a fiber which locally is of the form
    \[
        \Spec k[x,y,w]/(w^2 - x f(x,y),\, x),
    \]
    
    which is generically non-reduced of multiplicity~$2$; 
    its reduced structure is necessarily isomorphic to $F_p \simeq \bP^1$.

    If $\Theta$ is smooth at $q$, then by 
    Lemma~\ref{lem:at worst double point} we have 
    $k := (\Theta,F_p)_q \le 4$, and the branch divisor has a local equation 
    $x - y^k = 0$.  
    Thus the double cover has an $A_{2k-1}$-singularity.

    If $\Theta$ is singular at $q$, then by 
    Lemma~\ref{lem:at worst double point}, $q$ is a double point of $\Theta$, so 
    locally $\Theta$ is given by $x^2 - y^k = 0$. 
    \begin{itemize}
        \item If $F_p$ is not given by $x=0$, then after a coordinate change we may assume $F_p$ is given by $y=0$. In this case the branch divisor has a $D_k$-singularity, hence so does $X_G$.
        \item If $F_p$ is given by $x=0$, then the condition that $(Y,(\tfrac{1}{2}+\epsilon)R)$ is slc forces $k \le 3$, and the resulting singularity is of type $E_7$, hence $X_G$ has either a $D_4$ or an $E_7$-singularity at~$q$.
    \end{itemize}
\end{proof}

\begin{example}\textup{
    Let $R \subseteq \bP^1 \times \bP^1$ be a reduced divisor of class 
    $\mathcal{O}(4,4)$, consisting of the union of eight distinct fibers of the 
    two projections.  
    Then the double cover $f: X \to \bP^1 \times \bP^1$ branched along $R$ 
    is a K3 surface with $16$ $A_1$-singularities, and the composition of $f$ 
    with either projection 
    $p_i: \bP^1 \times \bP^1 \to \bP^1$ has four multiple fibers of multiplicity $2$.
}\end{example}

\begin{lemma}\label{lem:at worst double point 2}
    Assume that $Y_G$ is of type~\textup{III}, i.e.\ it is generically nodal along a curve $\Gamma$.
    On any irreducible component of the normalization of $Y_G$, let $q \in F_p$ be a singular point of $R$, and write
    \[
        \Theta := R|_G - F_p
    \]
    for the residual curve to $F_p$ in $R|_G$. Then:
    \begin{enumerate}
        \item $q$ is disjoint from $\Gamma$;
        \item $\Theta$ has at worst a double point at $q$; and
        \item if $\Theta$ is smooth at $q$, then the local intersection number 
        $(\Theta, F_p)_q \leq 2$.
    \end{enumerate}
\end{lemma}

\begin{proof}
    If $q \in \Gamma$, then $F_p$, $\Theta$, and $\Gamma$ all pass through $q$.
    Since $R|_G = F_p + \Theta$, the divisor $R$ contains three distinct branches
    through the node of $Y_G$ at $q$, and hence the pair 
    \[
        \bigl(Y_G, (\tfrac{1}{2}+\epsilon)R\bigr)
    \]
    is not slc at $q$. Thus $q \notin \Gamma$. The remaining assertions follow exactly as in 
    Lemma~\ref{lem:at worst double point}.
\end{proof}

By the same argument as in Theorem~\ref{thm:classification of fibers over nodes}, 
one obtains the following analogue.

\begin{prop}\label{prop:singularity of type III}
    Assume that $Y_G$ is of type~\textup{III}.  
    Then the morphism $X_G \to G$ has a double fiber over $p$, whose reduced structure 
    is a dumbbell curve.  
    Locally over $p \in G^\circ$:
    \begin{enumerate}
        \item If $\Theta$ is smooth at $q$, then $X_G$ has an 
        $A_{2k-1}$-singularity with $k \le 2$.
        \item If $\Theta$ is singular at $q$, then $X_G$ has either a 
        $D_k$-singularity or an $E_7$-singularity.
    \end{enumerate}
\end{prop}


\section{Compactification of moduli of hyperelliptic K3 surfaces}

An important application of Theorem~\ref{thm:elliptic with a bisection moduli} is the compactification of the moduli stack of hyperelliptic K3 surfaces.

\begin{defn}
    A \emph{polarized K3 surface} $(S,H)$ consists of a projective surface $S$ with at worst ADE singularities such that $\omega_S\sim \mtc{O}_S$ and $\oH^1(S,\mtc{O}_S)=0$, and an ample line bundle $H$.
\end{defn}

Recall that one has the following trichotomy for polarized K3 surfaces.

\begin{theorem}[\cite{May72,SD}]
Let $(S,L)$ be a \emph{primitively polarized} K3 surface with $(L^2)=2g-2$. Then exactly one of the following holds.
\begin{enumerate}
    \item \textup{(Generic case)} The linear system $|L|$ is very ample, and the morphism $
        \phi_{|L|}\colon S \hookrightarrow |L|^{\vee} \simeq \mathbb{P}^g$ embeds $S$ as a surface of degree $2g-2$ in $\mathbb{P}^g$. In this case, a general member of $|L|$ is a smooth non-hyperelliptic curve.
    
    \item \textup{(Hyperelliptic case)} The linear system $|L|$ is base-point free, and the induced morphism $\phi_{|L|}$ realizes $S$ as a double cover of a normal surface of degree $g-1$ in $\mathbb{P}^g$. In this case, a general member of $|L|$ is a smooth hyperelliptic curve, and $|2L|$ is very ample.
    
    \item \textup{(Unigonal case)} The linear system $|L|$ has a fixed component $E$, which is a smooth rational curve. The linear system $|L-E|$ defines a morphism
    $S \longrightarrow \mathbb{P}^g
    $ whose image is a rational normal curve. In this case, a general member of $|L-E|$ is a union of smooth elliptic curves, and $|2L|$ is base-point free.
\end{enumerate}
\end{theorem}

\begin{remark}
\textup{
The integer $g$ is called the \emph{genus} of the polarized K3 surface $(S,L)$, and it satisfies $g \geq 2$. When $g=2$, a general polarized K3 surface $(S,L)$ is realized as a double cover of $\mathbb{P}^2$ branched along a sextic curve; this case is exceptional among all genera.  
Therefore, for the remainder of this section, we will assume $g \geq 3$.
}
\end{remark}

Let $(S,L)$ be a hyperelliptic K3 surface of genus $g$. Then the Picard group $\Pic(S)$ contains a primitive rank-two sublattice
\[\Lambda^g \ \coloneqq\ 
\Lambda^{g}_{2,0} \ \coloneqq \ 
\begin{pmatrix}
2g-2 & 2 \\[2pt]
2 & 0
\end{pmatrix},
\]
generated by two effective divisor classes $L$ and $F$. The class $L$ defines a double cover
\[
S \longrightarrow \mathbb{F}_n,
\qquad 0 \leq n \leq 4,
\]
see e.g. \cite[p.~1]{Reid76}. The divisor $F$ is the pullback of the fiber class of $\mathbb{F}_n$, and hence it satisfies $F^2=0$ and is nef; it induces an elliptic fibration
\[
f \colon S \longrightarrow \mathbb{P}^1.
\] The \emph{branch divisor} $R_{\bF_n}$ on $\mathbb{F}_n$ is smooth and of the class
\[
4\mathfrak{e} + 2(n+2)\mathfrak{f},
\]
where $\mathfrak{e}$ and $\mathfrak{f}$ denote respectively the classes of the negative section and of a fiber of the ruling on $\mathbb{F}_n$. If moreover $(S,L)$ is a general hyperelliptic K3 surface, then $n=0$ when $g$ is even and $n=1$ when $g$ is odd, reflecting the parity constraint imposed by the discriminant form of $\Lambda^{g}_{2,0}$. By the Riemann-Roch theorem, the class
\[
B := L - \left\lfloor \frac{g}{2} \right\rfloor F
\]
is effective and satisfies
\[
B \cdot F = 2,
\]
and hence defines a bisection (not necessarily integral) of the elliptic fibration $f$.

\begin{lemma}\label{lem:KSBA-stable}
    With the notations as above. Let $R\subseteq S$ be the ramification locus of $S\rightarrow \bF_n$, and $F_{\sing}$ be the sum of the singular locus of $f$, counted via the canonical bundle formula. Then $(S,\epsilon_1 R+\epsilon_2 F_{\sing})$ is KSBA-stable for any $0< \epsilon_1\ll\epsilon_2\ll1$. 
\end{lemma}

\begin{proof}
    It suffices to show that $$K_S+\epsilon_1 R+\epsilon_2 F_{\sing}\sim_{\bQ} \epsilon_1 R+\epsilon_2 F_{\sing}$$ is ample. As $(R.F_{\sing})>0$ and $\epsilon_1\ll \epsilon_2$, then one has $(\epsilon_1 R+\epsilon_2 F_{\sing})^2>0$. For any integral curve $C\subseteq S$, if $C$ is not contained in a fiber of $f$, then $(\epsilon_1 R+\epsilon_2 F_{\sing}.C)>0$, because $(F.C)>0$ and $\epsilon_1\ll \epsilon_2$. If $C$ is a component of an $f$-fiber, then $$\textstyle (\epsilon_1 R+\epsilon_2 F_{\sing}.C)\ =\ \epsilon_1(R.C)\ =\ \frac{\epsilon_1}{2}(\pi^*R_{\bF_n}.C)\ =\ \frac{\epsilon_1}{2}(R_{\bF_n}.\pi_*C)\ = \ \frac{\epsilon_1}{2}(4\mtf{e}+2(n+2)\mtf{f}.m\mtf{f})  $$ for some $f>0$, which is positive. Therefore, the divisor $\epsilon_1 R+\epsilon_2 F_{\sing}$ is ample as desired.
\end{proof}

\begin{defn}
    For any number $0<c<1$, we set $H_c\coloneqq(1-c)L+c F\in \Lambda^g\otimes \bR$ and we say that $c$ is \emph{very irrational} if $H_c\notin \Lambda'\otimes \bR $ for any proper sublattice $\Lambda'\subsetneq \Lambda^g$. 
\end{defn}

 The sublocus of hyperelliptic K3 surfaces inside the moduli stack $\mathcal{F}_g$ of polarized K3 surfaces of genus $g$ need not be normal. The correct object to consider is the moduli stack of lattice polarized K3 surfaces, which is the normalization of the hyperelliptic locus; see Proposition~\ref{prop:normalization of the hyperelliptic divisor}.

\begin{defn}[\textup{cf.  \cite{Dol96,AE25}}]
    \textup{
    For any very irrational $0<c<1$, let $\mts{F}_{\Lambda^g,c}$ be the moduli stack  which sends a base scheme $T$ to 
\[
\left\{(f:\mts{S}\rightarrow T;\varphi)\left| \begin{array}{l} \mts{S}\to T\textrm{ is a proper flat morphism, each geometric fiber}\\ \textrm{$\mts{S}_{\bar{t}}$ is an ADE K3 surface, and $\varphi:\Lambda^g\longrightarrow\Pic_{\mts{S}/T}(T)$ is }\\ \textrm{a group homomorphism such that the induced map }\\ \textrm{$\varphi_{\bar{t}}:\Lambda^g\rightarrow \Pic(\mts{S}_{\bar{t}})$ is an isometric primitive embedding of}\\ \textrm{lattices and that $\varphi_{\bar{t}}(H_c)\in \Pic(\mts{X}_{\bar{t}})_{\mb{R}}$ is an ample class.} \end{array}\right.\right\}.
\] We call $\mts{F}_{\Lambda^g,c}$ a \emph{moduli stack of lattice polarized K3 surfaces}, where the lattice is $\Lambda^g$ and the polarization is given by $H_c=(1-c)L+c F.$}
\end{defn}

\begin{remark}\textup{
   In the definition of the moduli stack $\mathcal{F}_{\Lambda^g,c}$, the class $L$ is not required to be ample; only $H_c$ is assumed to be ample. Nevertheless, in the situations we shall consider, the class $L$ will in fact turn out to be ample.
}
\end{remark}

\begin{theorem}[\textup{cf. \cite[Proposition 5.4, Theorem 5.5]{AE25}}] The stack $\mts{F}_{\Lambda^g,c}$ is a separeted smooth DM stack, and is independent of the choice of $c$.
\end{theorem}

\begin{remark}
\textup{
More precisely, for any two parameters $0 < c \neq c' < 1$, the stacks $\sF_{\Lambda_g,c}$ and $\sF_{\Lambda_g,c'}$ are canonically isomorphic, although their universal families may differ. There is a natural wall--crossing structure: there exist finitely many critical values
\[
0 = c_0 < c_1 < \cdots < c_n = 1
\]
such that for each $i = 0, \ldots, n-1$, the universal family over $c \in (c_i, c_{i+1})$ is independent of the specific choice of $c$. The intervals $(c_i, c_{i+1})$ are called \emph{small cones}. Equivalently, the small cones are the connected components of the complement of the set of parameters $c$ for which there exists a vector 
\[
v \in \Lambda_{\mathrm{K3}} \setminus \Lambda^{\perp}
\quad\text{with}\quad v^2 = -2,
\]
such that $(H_c \cdot v) = 0$ and the lattice $\langle \Lambda, v \rangle \subseteq \Lambda_{\mathrm{K3}}$, is hyperbolic. See \cite[Definition~4.9, Proposition~4.14, Definition~5.2]{AE25}.
}
\end{remark}

\begin{lemma}\label{lem:nefness of F}\label{lem:no walls for moduli of K3}
    Let $0<\epsilon\ll 1$ be a very irrational number, and let 
    \[
    (S,(1-\epsilon)L+\epsilon F)\in \mts{F}_{\Lambda^g,\epsilon}
    \]
    be a $\Lambda^g$-polarized K3 surface of genus $g\geq 3$. Then the class $F$ is nef. In particular, $H_c$ is ample for any $0<c<1$, i.e. all the $c\in (0,1)$ are in the same small cone.
\end{lemma}

\begin{proof}
    By assumption $H_{\epsilon}=(1-\epsilon)L+\epsilon F$ is ample; in particular, $L$ is nef and big. Let $\pi:\widetilde{S}\to S$ denote the minimal resolution, so $\widetilde{S}$ is a smooth K3 surface. Denote by $\widetilde{L}$ and $\widetilde{F}$ the pullbacks of $L$ and $F$, respectively.

    Suppose that $F$ was not nef. Then $\widetilde{F}$ is also not nef, so there exists a connected smooth rational curve $\widetilde{C}\subset \widetilde{S}$ such that $C:=\pi(\widetilde{C})$ satisfies
    \[
    a := (\widetilde{C}\cdot \widetilde{L}) = (C\cdot L)\ge 0,
    \qquad
    b := (\widetilde{C}\cdot \widetilde{F}) = (C\cdot F) < 0.
    \]
    Since $H_\epsilon$ is ample, one has
    \[
    (1-\epsilon)a + \epsilon b > 0.
    \] Consider the sublattice of $\Pic(\widetilde{S})$ generated by $\widetilde{L}$, $\widetilde{F}$, and $\widetilde{C}$; its intersection matrix is
    \[
    \begin{pmatrix}
        2g-2 & 2 & a\\
        2 & 0 & b\\
        a & b & -2
    \end{pmatrix}.
    \]
    By the Hodge index theorem, its determinant must be positive. A direct computation shows that this determinant is
    \[
        (2-2g)b^2 + 8 + 4ab > 0.
    \]
    Since $g\ge 3$, this determinant is at most \(4(b(a-b)+2)\). Hence the only possible solution is \(a=0\) and \(b=-1\), contradicting the inequality \((1-\epsilon)a+\epsilon b>0\). Thus $F$ must be nef.
\end{proof}

\begin{corollary}\label{cor:isomorphic moduli of K3}
   For every fixed $g\ge 3$ and every $0<c<1$, the moduli stacks $\mtc{F}_{\Lambda^g,c}$ are canonically isomorphic, and their universal families are isomorphic. Moreover, for each $g\ge 3$ there is a canonical isomorphism
   \[
       \phi_g : \mts{F}_{\Lambda^g,c} \xrightarrow{\ \sim\ } \mts{F}_{\Lambda^{g+2},c},
   \]
   whose universal families of K3 surfaces are isomorphic and whose effect on the polarizing lattices is given by $L\mapsto L+F$ and $F\mapsto F$.
\end{corollary}

\begin{proof}
    Fix $g\ge 3$ and consider the moduli stack $\mts{F}_{\Lambda,\epsilon}$ with $0<\epsilon\ll 1$. By Lemma~\ref{lem:no walls for moduli of K3}, the divisor $H_c$ is ample for every $0<c<1$ and every $(S,H_{\epsilon})\in \mts{F}_{\Lambda,\epsilon}$. Therefore, all $0<c<1$ lie in the same generalized small chamber, and the corresponding stacks $\mtf{F}_{\Lambda,c}$ have isomorphic universal families.

    Now let $(S,(1-\epsilon)L+\epsilon F)\in \mts{F}_{\Lambda^g,\epsilon}$ with $g\ge 3$. Lemma~\ref{lem:no walls for moduli of K3} implies that $F$ is nef, hence $L+F$ and $(1-\epsilon)(L+F)+\epsilon F$ are ample. This defines a natural morphism
    \[
    \phi_g:\mts{F}_{\Lambda^g,\epsilon} \longrightarrow \mts{F}_{\Lambda^{g+2},\epsilon},
    \qquad
    (S,(1-\epsilon)L+\epsilon F) \longmapsto (S,(1-\epsilon)(L+F)+\epsilon F).
    \] Conversely, given $(S,(1-\epsilon)L+\epsilon F)\in \mts{F}_{\Lambda^{g+2},\epsilon}$, we claim that $L-F$ is nef, and hence the inverse morphism of $\phi_g$ is constructed. Since $(L-F)^2>0$, nefness implies that $(1-\epsilon)(L-F)+\epsilon F$ is ample, giving the inverse morphism. If $L-F$ were not nef, then for some curve $C$ we would have
    \[
        a := (L-F)\cdot C < 0,
        \qquad
        b := F\cdot C \ge 0.
    \]
    Repeating the intersection-matrix argument in the proof of Lemma~\ref{lem:no walls for moduli of K3} shows that the only numerical possibility is $a=-1$ and $b=0$, contradicting the fact that $L$ is nef. Thus $L-F$ must be nef, completing the proof.
\end{proof}

Therefore, it is reasonable to denote $\mts{F}_{\Lambda^g,c}$ simply by $\mts{F}_{\Lambda^g}$ for any $0<\epsilon<1$. Moreover, there are essentially only two stacks, depending on the parity of $g$. Since the results in the remainder of this section are insensitive to $g$, we henceforth write $\mts{F}_{\Lambda}$ for $\mts{F}_{\Lambda^g}$, regardless of parity.
\medskip

By the proof of Corollary \ref{cor:isomorphic moduli of K3}, for any $g\ge 5$ and any $
(S,(1-c)L+cF)\in \mts{F}_{\Lambda}$, the divisor $L$ is ample. We observe that that this fails for $g=3,4$. Indeed, for $g=3$, let $\pi:S\rightarrow \mathbb{F}_2$ be the double cover branched along a smooth divisor of class $4(\mathfrak{e}+2\mathfrak{f})$. Then $S$ is a K3 surface, and $L := \pi^*(\mathfrak{e}+2\mathfrak{f})$ is not ample, although
$(S,(1-c)L+c\,\pi^*\mathfrak{f})$ is a $\Lambda^3$-polarized K3 surface for $0<c<1$. Similarly, for $g=4$, let $\pi:S\longrightarrow \mathbb{F}_3$ be the double cover branched along a smooth divisor of class $4(\mathfrak{e}+2\mathfrak{f})$. Then $S$ is a K3 surface, and $L := \pi^*(\mathfrak{e}+3\mathfrak{f})$ is not ample, but
$(S,(1-c)L+c\,\pi^*\mathfrak{f})$
is a $\Lambda^4$-polarized K3 surface for $0<c<1$.

\medskip

For every $g \ge 3$, the hyperelliptic polarized K3 surfaces $(S,L)$ form a divisor in the moduli stack $\mathscr{F}_g$ of primitively polarized K3 surfaces of genus $g$. When $g = 4k+3 \ge 7$, the hyperelliptic locus has two components; one of these generically consists of double covers of $\bF_4$ branched along a divisor of the class $4\mtf{e} + 12\mtf{f}$, which is a disjoint union of the negative section and a trisection. The K3 surfaces parametrized by this component are always elliptic with a section, and we will not consider them here. For all other genera $g \ge 3$, the hyperelliptic locus is irreducible. We refer to this irreducible component as the \emph{hyperelliptic divisor} and denote it by $\mts{D}^g_{2,0}$.

\begin{remark}
\textup{
As stated in \cite[Page~2]{Reid76}, the hyperelliptic locus in $\sF_g$ has two components for all odd genera $g \ge 5$.  
However, when $g = 4k+1$, the polarization is not primitive, so this case does not lie in the moduli stack $\mathscr{F}_g$ of \emph{primitively} polarized K3 surfaces.
}
\end{remark}

\begin{prop}\label{prop:normalization of the hyperelliptic divisor}
For $g\ge 4$, the moduli stack $\mts{F}_{\Lambda}$ is isomorphic to the normalization $\mts{D}^{g,\nu}_{2,0}$ of the hyperelliptic divisor $\mts{D}^g_{2,0}$.
\end{prop}

\begin{proof}
We prove the statement for $g$ even; the odd case is identical. Let $0<\epsilon\ll 1$ be very irrational, and let 
\[
(\mts{S},(1-\epsilon)\cL+\epsilon\cF)\longrightarrow \sF_{\Lambda,\epsilon}=\sF_{\Lambda}
\]
be the universal family. Then $\cL$ is big and nef over $\sF_{\Lambda}$; taking the ample model of $\cL$ (which is an isomorphism for $g\ge 5$) yields a natural morphism
\[
\sF_{\Lambda} \longrightarrow \mts{F}_{g}
\]
by the universality of $\mts{F}_g$, whose image is contained in $\mts{D}^g_{2,0}$. Since $\mts{F}_{\Lambda}$ is smooth, there exists a morphism
\[
\alpha:\sF_{\Lambda}\longrightarrow \mts{D}^{g,\nu}_{2,0}.
\]
We claim that $\alpha$ is representable, birational, and finite; the result then follows from Zariski’s main theorem for DM stacks.

\noindent\textbf{Quasi-finiteness.}
Let $[(\overline{S},\overline{L})]\in \mts{D}^g_{2,0}$ be a hyperelliptic K3 surface. Let $\pi:\widetilde{S}\to \overline{S}$ be the minimal resolution. If 
\[
[(S,(1-\epsilon)L+\epsilon F)]\in \mts{F}_{\Lambda,\epsilon}
\]
lies above it, then $\widetilde{S}$ factors as 
\[
\widetilde{S}\to S\to \overline{S},
\]
contracting finitely many $(-2)$-curves; thus there are finitely many choices of $S$. The class $L$ is the pullback of $\overline{L}$. By \cite[Proposition 11.1.3]{Huy16}, a K3 surface has only finitely many elliptic fibrations up to isomorphism, so there are finitely many choices of $F$. Hence $\alpha$ is quasi-finite.

\noindent\textbf{Birationality.}
A general $(S,L)\in \mts{D}^g_{2,0}$ is a double cover of $\mathbb{F}_1$, and $L = \pi^*(\mathfrak{e} + (\frac{g}{2}-1)\mathfrak{f})$. If $[(S,(1-c)L+cF)]$ lies in the fiber of $\alpha$ over $[(S,L)]$, then $F$ must be the pullback of $\mathfrak{f}$. Hence $\alpha$ is generically injective. Since both stacks are irreducible of dimension $18$, then $\alpha$ is birational.

\noindent\textbf{Representability.}
It suffices to check that stabilizers inject. As $\epsilon$ is very irrational, any stabilizer of $(S,(1-\epsilon)L+\epsilon F)$ preserves $L$ and hence acts on the ample model of $L$. If the induced automorphism is trivial, then the stabilizer is trivial, proving the representability.

\noindent\textbf{Properness.} Finally, we prove that $\alpha$ is proper, and therefore finite. Since both stacks are separated, it suffices to verify the existence part of the valuative criterion. Let $h:(\overline{\mts{S}},\overline{\mtc{L}})\longrightarrow (0\in B)$ be a one–parameter family of hyperelliptic K3 surfaces of genus $g$ such that there exists a commutative diagram
\[
\xymatrix{
 & (\mts{S}^{\circ},(1-c)\mtc{L}^{\circ}+c\mtc{F}^{\circ}) 
 \ar[rr]^{g^\circ} \ar[dr]_{\widetilde{h}^{\circ}} 
 & & (\overline{\mts{S}}^{\circ},\overline{\mtc{L}}^{\circ}) 
 \ar[dl]^{h^{\circ}} \\
 & & B^{\circ} &
}
\]
where $B^{\circ}=B\setminus\{0\}$, and $g^\circ$ is the ample model morphism over $B^{\circ}$ with respect to $\mtc{L}$. We may assume that $g^\circ$ is an isomorphism. Furthermore, possibly after a finite base change, there exists a section of $\mtc{F}^{\circ}$ over $B^{\circ}$; we continue to denote it by $\mtc{F}^{\circ}$. Let $\ove{\mtc{F}}$ be the closure of $\mtc{F}^{\circ}$ inside $\overline{\mts{S}}$ as a subscheme. Take a $\mathbb{Q}$-factorialization of $\ove{\mts{S}}$ and run a relative $\ove{\mtc{F}}$-MMP over $\ove{\mts{S}}$. This yields a family of K3 surfaces 
\[
\sS\longrightarrow B,
\]
a birational morphism 
$g:\sS\longrightarrow \ove{\sS}$ which is isomorphic in codimension one, and a $\mathbb{Q}$-Cartier divisor $\cF$ such that the divisor 
$(1-\epsilon){g}^{*}\mtc{L}+\epsilon{\cF}$ is relatively ample over $B$. We claim that $\mtc{F}$ is a Cartier divisor. Then every fiber of the pair $
\bigl({\sS},(1-\epsilon){g}^*\ove{\mtc{L}}+\epsilon{\cF}\bigr)$
over $b\in B$ is a $\Lambda$-polarized K3 surface: this follows because the general fiber has this property and the intersection numbers remain constant throughout the family.
The line bundle $\overline{\cL}$ is ample and base-point free over $B$, and it induces a double cover 
\[
\overline{\sS} \longrightarrow \overline{\sT},
\]
where the general fiber of $\overline{\sT} \to B$ is a Hirzebruch surface $\bF_n$, and the special fiber over $0 \in B$ is either $\bF_n$ or $\bP(1,1,m)$ for some $m \le 4$.  

In the former case, there exists a prime divisor $\overline{\cG}$ on $\overline{\sT}$, Cartier over $B$, such that for a general $b \in B$ its restriction to the fiber is the fiber class of the fibration $\bF_n \to \bP^1$. The divisor $\cF$ is then the pullback of $\overline{\cG}$ along
\[
\sS \longrightarrow \overline{\sS} \longrightarrow \overline{\sT},
\]
and therefore is Cartier. In this situation, $\overline{\cF}$ is already Cartier, and no $\bQ$-factorialization or MMP is required.

In the latter case, we take a $\bQ$-factorialization of $\overline{\sT}$, denoted by $\sT$, such that the special fiber over $0 \in B$ becomes isomorphic to $\bF_n$, and the strict transform of $\overline{\cG}$ on $\sT$, denoted by $\cG$, is ample over $\overline{\sT}$. Composing with the morphism $\sS \to \overline{\sS}$ yields a generically $2\!:\!1$ rational map
\[
\sS \dashrightarrow \sT.
\]
Let $\cR$ denote the closure of the ramification locus of this map. Let $\wt{\sS}$ be the normalization of the double cover of $\sT$ branched along $\cR$. Then $\wt{\sS}$ is isomorphic to $\sS$ in codimension one, and both are ample models of the $\bQ$-Cartier divisor $(1-\epsilon)\cL + \epsilon \cF$. Hence they are in fact isomorphic. Consequently, $\cF$ is the pullback of the Cartier divisor $\cG$ on $\sS$, and is therefore Cartier.
\end{proof}

\begin{remark}
\textup{
The above proposition does not hold for $g=3$: the morphism $\mts{F}_{\Lambda^3}\to \mts{D}^{3}_{2,0}$ is even not birational. Indeed, it is generically $2\!:\!1$: a general $(S,L)\in \mts{D}^{3}_{2,0}$ is a double cover of $\mathbb{P}^1\times\mathbb{P}^1$, with $L$ the pullback of $\mathcal{O}(1,1)$. However, there are two choices for $F$, namely the pullbacks of $\mathcal{O}(1,0)$ and $\mathcal{O}(0,1)$.
}
\end{remark}

Let $$\pi:(\sS,(1-c)\cL+c\cF)\ \longrightarrow \  \mts{F}_{\Lambda}$$ be the universal family. Then $\mtc{L}$ induces a double cover \[\begin{tikzcd}[ampersand replacement=\&]
	{\sS} \&\& {\sY} \\
	\& {\mts{F}_{\Lambda}}
	\arrow[from=1-1, to=1-3]
	\arrow[from=1-1, to=2-2]
	\arrow[from=1-3, to=2-2]
\end{tikzcd},\] which fiberwise is the aforementioned ramified cover $S\rightarrow \bF_n$. Denote by $\mts{R}$ the ramification locus. As the divisor $\mtc{F}$ is relatively nef over $\mts{F}_{\Lambda}$ by Lemma \ref{lem:nefness of F}, it induces an elliptic fibration \[\begin{tikzcd}[ampersand replacement=\&]
	{\sS} \&\& {\mts{C}\coloneqq \bP_{\mts{F}_{\Lambda}}(\pi_*\mtc{F})} \\
	\& {\mts{F}_{\Lambda}}
	\arrow[from=1-1, to=1-3]
	\arrow[from=1-1, to=2-2]
	\arrow[from=1-3, to=2-2]
\end{tikzcd}.\] Denote by $\mtc{F}_{\sing}\subseteq \mts{S}$ the closure of the sum of the singular fibers. Then by Lemma \ref{lem:KSBA-stable}, the family $(\sS,\epsilon_1 \sR +\epsilon_2 \cF_{\sing})\rightarrow \mts{F}_{\Lambda}$ is KSBA-stable, and hence there exists a natural morphism $\beta:\mts{F}_{\Lambda}\rightarrow \mtc{M}^{\KSBA}_{\textup{ell},2}(\epsilon_1,\epsilon_2)$, which is dominant. Let $\mtc{M}^{\KSBA,\nu}_{\textup{ell},2}(\epsilon_1,\epsilon_2)$ be the normalization of $\mtc{M}^{\KSBA}_{\textup{ell},2}(\epsilon_1,\epsilon_2)$, and $$\beta^\nu:\mts{F}_{\Lambda}\rightarrow \mtc{M}^{\KSBA,\nu}_{\textup{ell},2}(\epsilon_1,\epsilon_2)$$ be the morphism induced by $\beta$.
\begin{prop}
    The morphism $\beta^\nu$ is an open immersion.
\end{prop}

\begin{proof}
     Since both $\mts{F}_{\Lambda}$ and $\mtc{M}^{\KSBA,\nu}_{\textup{ell},2}(\epsilon_1,\epsilon_2)$ are separated DM stacks of finite type over $\bC$, Zariski's main theorem reduces the claim to proving that $\beta^\nu$ is representable, quasi-finite, and birational.

    Let $(S,\epsilon_1R+\epsilon_2F_{\sing})$ be a general object of $\mtc{M}^{\KSBA}_{\textup{ell},2}(\epsilon_1,\epsilon_2)$ and consider any $(S,(1-c)L+cF)$ mapping to it. For such $S$, the Picard rank is $\rho(S)=2$, and there is a unique elliptic fibration $f:S\to \bP^1$ for which $F_{\sing}$ is a multiple of the fiber class. Hence the class $F$ is uniquely determined. Moreover, $f$ admits a smooth bisection, which induces a double cover $S \to \bF_n$. Since $\rho(S)=2$, every bisection of $f$ is the pullback of a section of $\bF_n$, and they all induce the same double cover $S \to \bF_n$. Therefore the polarization class $L$ is uniquely determined as well, and $\beta^\nu$ is birational.

    Next we verify representability. For any $(S,(1-c)L+cF)\in \mts{F}_{\Lambda}$, an automorphism $\sigma$ fixes $L$, and in particular it fixes the ramification locus $R$ of the double cover induced by~$L$. Similarly, since $\sigma$ fixes $F$, it also fixes $F_{\sing}$. This yields a natural homomorphism
    \[
        \Aut(S,(1-c)L+cF)\longrightarrow \Aut(S,\epsilon_1R+\epsilon_2F_{\sing}),
    \]
    which is injective because both groups are subgroups of $\Aut(S)$. Hence $\beta$ is representable, and therefore so is $\beta^\nu$.

    Finally, to check that $\beta$ is quasi-finite, let $(S,\epsilon_1R+\epsilon_2F_{\sing})$ be in the image of $\beta$. A preimage under $\beta$ consists of $S$ together with the class $F$, which is uniquely determined by $F_{\sing}$. Thus it remains to show that there are only finitely many possibilities for the choice of $L$. Let $\wt S\to S$ be the minimal resolution and let $\wt L$ be the pullback of $L$. Since $\wt L$ realizes $\wt S$ as a double cover of one of the Hirzebruch surfaces $\bF_0,\dots,\bF_4$, and since $\wt L$ is the pullback of a unique line bundle on $\bF_n$, it follows that there are only finitely many possibilities for $\wt L$, and hence for $L$. Therefore $\beta$ is quasi-finite.

    Having established representability, quasi-finiteness, and birationality, Zariski's main theorem implies that $\beta^\nu$ is an open immersion.
\end{proof}

\begin{corollary}[Compactification of the moduli of hyperelliptic K3 surfaces]\label{cor:HE K3 surfaces}
For any genus $g \geq 3$, there exists a Deligne–Mumford stack $\mtc{M}^{\KSBA}_{\textup{ell},2}(\epsilon_1,\epsilon_2)$ whose objects are described in Theorems~\ref{thm:elliptic with a bisection moduli}, \ref{thm:Singularities over the Smooth Locus I}, \ref{thm:Singularities over smooth locus II}, and~\ref{thm:Singular fibers over nodes}.  
The normalization of this stack is a compactification of the normalization of the hyperelliptic divisor $\mts{D}^{g}_{2,0}$ in the moduli stack $\mts{F}_g$ of polarized K3 surfaces of genus $g$.
\end{corollary}

\section{Moduli of Weierstrass fibrations}

In this section, we use the results in \cite{ISZ25} to classify the singular objects parametrized by the KSBA moduli stack for Weierstrass fibrations.

\subsection{Geometry of Weierstrass fibrations}

Let $(X,S)\rightarrow C$ be a general Weierstrass fibration. We begin by observing that the linear series $|2S|$ is base-point free. Therefore, the doubled section $2S$ induces a double cover $$\textstyle \pi:\  X\ \longrightarrow \ Y:=\Proj_C\oplus_{m\geq0}\Sym^mf_*\mtc{O}_X(2S)$$ over $C$, where 
\begin{enumerate}
    \item every $\pi$-fiber is a $\bP^1$, and fiberwise it is a standard hyperelliptic involution quotient from an elliptic curve to $\bP^1$ induced by a $g^1_2$;
    \item the image of $S$ is a section of $g:Y\rightarrow C$, which we denote by $\Sigma$;
    \item the ramification divisor $R\subseteq Y$ (resp. $R_X\subseteq X$) of $\pi$ is a $4$-section of $g$ (resp. $f$) which satisfies that $\pi^*R= 2R_X$;
    \item $\Sigma$ is an irreducible component of $R$;
\end{enumerate}
The surface $Y$ is a ruled surface over $C$, and hence one can identify $Y$ with $\bP\mtc{E}$ for some vector bundle $\mtc{E}$ of rank $2$ on $C$. Moreover, one can assume that $\oH^0(C,\mtc{E})\neq0$ and $\oH^0(C,\mtc{E}\otimes \mtc{L})=0$ for any $\mtc{L}\in \Pic(C)$ with $\deg \mtc{L}<0$. Let $\sigma \in \oH^0(Y,\mtc{O}_Y(1))\simeq \oH^0(C,\mtc{E})$ be a section, where $\mtc{O}_Y(1)=\mtc{O}_{\bP\mtc{E}}(1)$ under the identification $$\Pic(Y)\ = \ \Pic (\bP\mtc{E}) \ = \ \bZ[\mtc{O}_{\bP\mtc{E}}(1)]\oplus g^*\Pic(C).$$ 
If $X$ is smooth, then $R$ is a smooth curve, and hence $R=\Sigma\sqcup T$ for some smooth divisor $T\sim 3\sigma+g^*D$, where $D$ is a divisor on $C$. The letter $T$ stands for ``trisection'' since $T\rightarrow C$ is of degree $3$.

\subsection{KSBA-moduli stack of Weierstrass fibrations}

Let $0<\epsilon\ll 1$ be a sufficiently small rational number, and let 
\[
\vec{a}=(a_1,\ldots,a_n), \qquad 0 < a_i \le 1.
\]
Consider a Weierstrass fibration
\[
(X_\eta,\ \epsilon S_\eta + \vec{a}F_\eta)\ \to\ C_\eta
\]
over the generic point $\eta$ of the spectrum of a DVR $A$, whose fibers are at worst irreducible nodal curves, and such that the pair $(X_\eta,\epsilon S_\eta+\vec{a}F_\eta)$ is KSBA-stable. Here:
\begin{itemize}
    \item $S_\eta$ is a section of $X_\eta\rightarrow C_\eta$;
    \item $p_\eta = p_{\eta,1}+\cdots+p_{\eta,n}$ is a sum of marked points on $C_\eta$, and 
    \[\textstyle
        \vec{a}p_\eta := \sum_{i=1}^n a_i p_{\eta,i};
    \]
    \item $F_\eta = F_{\eta,1}+\cdots+F_{\eta,n}$ is the corresponding sum of fibers, and 
    \[\textstyle
        \vec{a}F_\eta := \sum_{i=1}^n a_i F_{\eta,i}.
    \]
\end{itemize}
Denote by $\ove{\mtc{W}}^{\KSBA}(\epsilon,\vec{a})$ the KSBA moduli stack which generically parametrizes pairs of the form $(X_\eta,\epsilon S_\eta+\vec{a}F_\eta)$.

\begin{theorem}\label{thm_W_fib_with_a_section}
Let $\mathbf{M}_\eta$ be the moduli part of the canonical bundle formula for $X_\eta \to C_\eta$. Then the following statements hold:
\begin{enumerate}
    \item[\textup{(1)}] After a finite base change of $\Spec A$, the KSBA-stable extension $(X,\epsilon S+\vec{a}F)$ of $(X_\eta,\epsilon S_\eta+\vec{a}F_\eta)$ over $\Spec A$ admits a fibration over a generalized pair
    \[\textstyle
        (C,\vec{a}p+\bfM)\to\Spec A,
    \]
    which is a family of nodal pointed curves with $K_C+\vec{a}p+\bfM$ ample, and whose generic fiber is $ (C_\eta,\vec{a}p_\eta+\mathbf{M}_\eta)$.
    \item[\textup{(2)}] The family $C\to\Spec A$ is the coarse space of a family of twisted curves $\cC \to \Spec A$ equipped with a fibration $(\cX,\cS) \to \cC$, whose geometric fibers $(\cX_p,\cS_p)$ have equation
    \[
    \bigl( y^2x = x^3 + axz^2 + bz^3,\ [0,1,0] \bigr),
    \]
    and such that the induced morphism on coarse spaces is $(\cX,\cS)\to\cC$ giving $(X,S)\to C$.

    \item[\textup{(3)}] The fibers of $(X,S)\to C$ over codimension-one points of $C$ have at worst nodal singularities.
\end{enumerate}
\end{theorem}
\noindent The following diagram summarizes the situation:
    \[
    \begin{tikzcd}[ampersand replacement=\&]
        {(X_\eta,\epsilon S_\eta + \vec{a}F_\eta)} \& {(X,\epsilon S+\vec{a}F)} \&\& {(\cX,\epsilon\cS+\vec{a}\cF)} \\
        {(C_\eta,\vec{a}p_\eta+\bfM)} \& {(C,\vec{a}p+\bfM)} \&\& {\cC} \\
        \eta \& {\Spec A}
        \arrow[from=1-1, to=1-2]
        \arrow[from=1-1, to=2-1]
        \arrow[from=1-2, to=2-2]
        \arrow["{\textup{cms}}"', from=1-4, to=1-2]
        \arrow["{\textup{Weierstrass}}"{description}, from=1-4, to=2-4]
        \arrow[from=2-1, to=2-2]
        \arrow[from=2-1, to=3-1]
        \arrow[from=2-2, to=3-2]
        \arrow["{\textup{cms}}"{description}, from=2-4, to=2-2]
        \arrow[from=2-4, to=3-2]
        \arrow[hook, from=3-1, to=3-2]
    \end{tikzcd}
    \] Theorem \ref{thm_W_fib_with_a_section} describes the KSBA-stable limits of Weierstrass fibrations.
\subsection{Proof of Theorem \ref{thm_W_fib_with_a_section}}

Take the quotient of $X_{\eta}$ by the involution induced by the linear system $|2S_{\eta}|$, and denote the result by $(Y_{\eta}, \tfrac{1}{2}R_{\eta}).$ Then for any $0<\epsilon \ll 1$, the pair $(X_{\eta},\, 2\epsilon S_{\eta} + \vec{a}\,F_{\eta})$ is KSBA–stable if and only if the pair
$$\textstyle \big(Y_{\eta},\, \tfrac{1}{2}R_{\eta} +\epsilon\Sigma_{\eta}+\vec{a}\,F_{\eta}\big)\ = \ \big(Y_{\eta},\, \tfrac{1}{2}T_{\eta} +(\frac{1}{2}+\epsilon)\Sigma_{\eta}+\vec{a}\,F_{\eta}\big)$$ is KSBA–stable, where we abuse notation by denoting the marked fibers of $X_{\eta} \to C_{\eta}$ and those of $Y_{\eta} \to C_{\eta}$ uniformly by $F_{\eta}$. Moreover, we denote by $R_\eta$ the ramification locus, which can be written as $R_\eta = \Sigma_\eta +T_\eta$, where $\Sigma_\eta$ is a section and $T_\eta$ a 3-section.

To find the KSBA–stable limit over $\Spec A$, we apply Corollary~\ref{cor_to_take_ksba_limit_one_can_take_quotient_first}. 
Namely, we obtain the KSBA–stable extension of $(Y_\eta,\, (\tfrac{1}{2}+\epsilon)\Sigma_\eta + \tfrac{1}{2}T_\eta + \vec{a}F_\eta)$, denoted by $(Y,\, (\tfrac{1}{2}+\epsilon)\Sigma + \tfrac{1}{2}T + \vec{a}F)$, by first taking the KSBA–stable limit of 
\[\textstyle 
\big(Y_\eta,\, (\tfrac{1}{2}+\epsilon_1)(\Sigma_\eta + T_\eta) + \vec{a}F_\eta + \epsilon_2 F'_\eta\big)
\]
where $F'_{\eta}$ is the sum of all fibers for which the support of the restriction of $R$ to them has fewer than four points; and then decreasing the coefficient of $T$ to $\frac{1}{2}$ and decreasing $\epsilon_2$ to $0$. The procedure has two steps.
\begin{enumerate}
    \item The first step is preparatory: we study the pair $(Y,\Sigma)$ in more detail; see \S \ref{sec:step 1}.
    \item In the second step, we construct the desired stable limits by first taking the canonical model of 
\[
(Y,\, (\tfrac{1}{2}+\epsilon_1)\Sigma + \tfrac{1}{2}T + \vec{a}F + \epsilon_2 F'),
\]
and then letting $\epsilon_2$ decrease to $0$; see \S \ref{sec:step 2}.
\end{enumerate}

 \subsubsection{Step 1: the family $(Y,\Sigma)\to \spec A$}\label{sec:step 1}

By Theorem \ref{thm: main theorem of ISZ25}, the pair $(Y_{\eta}, (\tfrac{1}{2}+\epsilon_1)D_{\eta} + \vec{a}F_{\eta} + \epsilon_2 F'_{\eta})$ extends to a KSBA–stable family \[\begin{tikzcd}[ampersand replacement=\&]
	{\big(Y,(\frac{1}{2}+\epsilon_1)D+\vec{a}F+\epsilon_2 F'\big)} \&\& C \\
	\& {\Spec A}
	\arrow[from=1-1, to=1-3]
	\arrow[from=1-1, to=2-2]
	\arrow[from=1-3, to=2-2]
\end{tikzcd}.\] Here, $0 < \epsilon_1 \ll \epsilon_2 \ll 1$ are two rational numbers. Let $\Sigma$ be the closure of $\Sigma_\eta$ and $T$ the closure of $D_\eta \setminus \Sigma_\eta$. We have the following key observation.

\begin{lemma}\label{lem:Q-Cartierness of T and Sigma}
    The divisor $\Sigma$ is $\bQ$-Cartier.
\end{lemma}

\begin{proof}
By Corollary \ref{cor:fiber over the node}, away from finitely many smooth points $x_1,\ldots,x_r \in C,$ the map 
$\pi \colon (Y,\tfrac{1}{2}D) \to C$ is the induced morphism between the coarse moduli spaces of $$(\cY,\tfrac{1}{2}\cD)\ \longrightarrow \ \cC,$$
where $\cC$ is a family of twisted curves over $\Spec A$, and the universal family comes from a morphism $\cC\to \sP_1$. The fibers of $\pi$ are either $\bP^1$ or the nodal union $\Gamma$ of two copies of $\bP^1$. Moreover, $\cD$ intersects $\Gamma$ at the smooth locus of $\Gamma$. In particular, since $\cC$ is $\bQ$-factorial, the space $\cY$ is also $\bQ$-factorial in a neighborhood of $\cD$. Thus $Y$ is $\bQ$-factorial near $D$, and away from the points $x_1,\ldots,x_r$.

It remains to show that $Y$ is $\bQ$-factorial over a neighborhood of each $x_i$. Let $G_i$ denote the irreducible component of the special fiber $C_0$ that contains $x_i$. In cases~(1) and~(2) of Theorem \ref{thm: main theorem 2 of ISZ25}, the surface $G_i \times_C Y$ is smooth over $x_i$. In case~(3), the surface $G_i \times_C Y$ has a unique singularity of the form
\[
    \Spec\!\bigl(k[\![x,y,s,t]\!]/(xy - s^n)\bigr)
\]
lying over $x_i$, for some integer $n$. Here $t$ is a uniformizer of the DVR $A$, while $s$ is the pullback of a local equation of $G_i \subset C$ at $x_i$. But this singularity is $\bQ$-factorial; hence $Y$ is $\bQ$-factorial in a neighborhood of $x_i$.
\end{proof}

\begin{lemma}\label{lem:T and Sigma are disjoint}
     The divisors $\Sigma$ and $T$ are disjoint.
\end{lemma}

\begin{proof}
 For the generic Weierstrass fibration $(X_{\eta},S_{\eta})\rightarrow C_{\eta}$, every fiber is either a smooth elliptic curve or an integral nodal curve, and $S$ is away from singularities of fibers. Thus the double cover $Y_{\eta}\rightarrow X_{\eta}$ is branched along $\Sigma_{\eta}$ of ramification index $1$, and $\Sigma_{\eta}$ is disjoint from $T_{\eta}$ on $X_{\eta}$. Since both $\Sigma$ and $T$ are $\bQ$-Cartier by Lemma \ref{lem:Q-Cartierness of T and Sigma} and the fact that $\epsilon_1$ is small enough, then \[(\Sigma_\eta.T_\eta) \ = \  (\Sigma_0.T_0) \ = \ 0,\]
which implies that either $T$ and $\Sigma$ do not intersect, or they intersect along a curve $\Gamma$. By inversion of adjunction, the pair $\big(Y,(\frac{1}{2}+\epsilon_1)(T+\Sigma)+\vec{a}F+\epsilon_2 F'+Y_0\big)$ is log canonical, and thus the latter case cannot occur: otherwise, it would not be log canonical along $\Gamma$. Therefore, $\Sigma$ and $T$ are disjoint.
\end{proof}

\begin{remark}\textup{
In what follows, the fibers $\vec{a}F$ do not play any role. To simplify notation, we will therefore omit them. More precisely, we will focus on the case $\vec{a}F = 0$; the case where $\vec{a}F \neq 0$ is entirely analogous.}
\end{remark}

\subsubsection{Step 2: KSBA-stable extension $(Y,(\frac{1}{2} + \epsilon)\Sigma +\frac{1}{2} T)\rightarrow \Spec A$.}\label{sec:step 2}   We now decrease the coefficients in 
\[
\bigl(Y,\, (\tfrac{1}{2}+\epsilon_1)R + \epsilon_2 F\bigr)
\ =\  \bigl(Y,\, (\tfrac{1}{2}+\epsilon_1)\Sigma + (\tfrac{1}{2}+\epsilon_1)T + \epsilon_2 F\bigr)
\]
to obtain 
\[
\bigl(Y,\, (\tfrac{1}{2}+\epsilon_1)\Sigma + \tfrac{1}{2}T\bigr),
\]
and take its log canonical model over $\Spec A$. Throughout this process, we keep track of the birational transformations in order to describe the geometry of the log canonical model of 
\(\bigl(Y,\, (\tfrac{1}{2}+\epsilon_1)\Sigma + \tfrac{1}{2}T\bigr)\).

We proceed in two steps: first, we decrease the coefficient of $T$ to $\tfrac{1}{2}$ and take the canonical model over $C$; then, we decrease the coefficient $\epsilon_2$ of $F$ to $0$. We begin with the following, which is a consequence of Lemma~\ref{lem:Q-Cartierness of T and Sigma}.

\begin{corollary}\label{cor_W_remains_lc_when_sigma_has_coef_1}
    The threefold pair $(Y,\Sigma+ \epsilon_2 F')$ is log canonical.
\end{corollary}
\begin{proof}
We first show that the pair $(Y,\, \Sigma + \epsilon_2 F')$ is slc over the smooth points of $C_0$, the central fiber of $C \to \Spec A$. Since the fibers $F$ intersect $\Sigma$ transversely, it suffices to prove that $(Y,\Sigma)$ is log canonical.

Let $G$ be an irreducible component of $C_0$, and let $x \in G$ be a point which is not a node of $C_0$. Denote by $F_x$ the fiber over $x$. Then 
\[
(F_x \cdot \Sigma) = 1,
\]
since $F_x$ is numerically equivalent to a fiber of the generic fiber. In particular, $\Sigma \to C$ remains a section and is therefore smooth over such $x$. As $\Sigma$ does not intersect $T$ by Lemma~\ref{lem:T and Sigma are disjoint}, and $Y_G:=Y|_G \to G$ is smooth over $x$ in cases~(1) and~(2) of Theorem~\ref{thm: main theorem 2 of ISZ25}, the desired local log canonicity follows. Hence, to show that $(Y,\, \Sigma + \vec{a}F + \epsilon_2 F')$ is slc over the smooth points of $C_0$, it remains only to analyze $\Sigma$ in case~(3) of Theorem~\ref{thm: main theorem 2 of ISZ25}.

For this remaining case, observe that $\Sigma$ does not meet the horizontal double locus of $Y_G$. Indeed,
\begin{enumerate}
    \item $\Sigma$ is $\mathbb{Q}$-Cartier, so the pair $(Y, t\Sigma)$ is locally stable for every $0 < t \leq \tfrac{1}{2}$. In particular, the reduced pair
    \[
    \bigl(Y_G,\, t\Sigma|_{Y_G} \bigr)
    \]
    is slc for all $t \in [0, \tfrac{1}{2}]$;
    \item as $\Sigma \to C$ is a section, the intersection $\Sigma \cap F_x$ consists of a single reduced point.
\end{enumerate} If $\Sigma$ met the double locus of $Y_G$, then \'{e}tale locally near such a point, $Y_G$ would have two irreducible components, with $\Sigma$ lying on exactly one of them. Consider $Y_G^n = (\Gamma_1,\Delta_1)\sqcup (\Gamma_2,\Delta_2)$ where $\Gamma_i$ are the irreducible components of the normalization of $Y_G$, and $\Delta_i$ the preimages of the double locus. The proper transform $\Sigma_1$ of $\Sigma$ is contained in $\Gamma_1$ but not $\Gamma_2$, and if $\Sigma$ intersected the double locus, then there would be a point $p\in \Delta_1$ such that the different $\operatorname{Diff}_{\Delta_1}((\Gamma_1,\Delta_1+t\Sigma_1)$ depends on $t$. On the other hand, as $\Sigma_1$ is not contained in $\Gamma_2$, there is no point in $\operatorname{Diff}_{\Delta_2}((\Gamma_2,\Delta_2+t\Sigma_1)$ which depends on $t$. This contradicts Koll\'ar's gluing theorem.

We now analyze the pair over the nodes of $C$. By the construction of the ruled model, after possibly restricting $C$ to a neighborhood of its nodal locus, the family 
\[
(Y,\, \tfrac{1}{2}\Sigma + \tfrac{1}{2}D) \to C
\]
is obtained by pulling back the universal family from a twisted curve $\cC \to \sP_1$ and then passing to coarse spaces. Since $\Sigma$ is a section and is disjoint from $T$, every fiber has the form
\[
\bigl(\sR,\, \tfrac{1}{2} p_\Sigma + \tfrac{1}{2}(p_1 + p_2 + p_3)\bigr),
\]
with $p_\Sigma \neq p_i$ for all $i$. Such fibers satisfy that $(\sR,\, p_\Sigma + \tfrac{1}{2}(p_1 + p_2 + p_3))$ is slc, and therefore the original pair is slc over the nodal locus as well.
\end{proof}

Due to the presence of the section $\Sigma$, if $G \subseteq C_0$ is an irreducible component such that $Y_G$ falls into case~(3) of Theorem~\ref{thm: main theorem 2 of ISZ25}, then $Y_G$ necessarily has two irreducible components: one meeting $\Sigma$ and one disjoint from $\Sigma$. Moreover, we have the following.

\begin{corollary}
The pair 
\[
\bigl(Y,\; (\tfrac{1}{2}+\epsilon_1)\Sigma + \tfrac{1}{2}T + \epsilon_2 F\bigr)
\]
is relatively minimal over $C$.
\end{corollary}

\begin{proof}
    This is because $0<\epsilon_1\ll\epsilon_2$ and $(Y,\frac{1}{2}(\Sigma+T))$ is relatively minimal over $C$.
\end{proof}


Thus, taking the log canonical model over $C$
\[
\left(Y,\, (\tfrac{1}{2}+\epsilon_1)\Sigma + \tfrac{1}{2}T + \epsilon_2 F\right)
\ \longrightarrow\ 
\left(Y^{(1)},\, (\tfrac{1}{2}+\epsilon_1)\Sigma^{(1)} + \tfrac{1}{2}T^{(1)} + \epsilon_2 F^{(1)}\right)
\]
contracts precisely those irreducible components of the surface pairs appearing in case~(3) of Theorem~\ref{thm: main theorem 2 of ISZ25} that do not meet $\Sigma$; see Figure \ref{fig:Y1 over C}. For a detailed description of the geometry of these components, see \cite[Proposition~4.6 and Figure~10]{ISZ25}. Furthermore, the fibers of  $\pi^{(1)} : Y^{(1)} \to C$ over the codimension-one points of $C$ are isomorphic to $\mathbb{P}^1$.

\begin{figure}
    \centering
    \includegraphics[width=.6\linewidth]{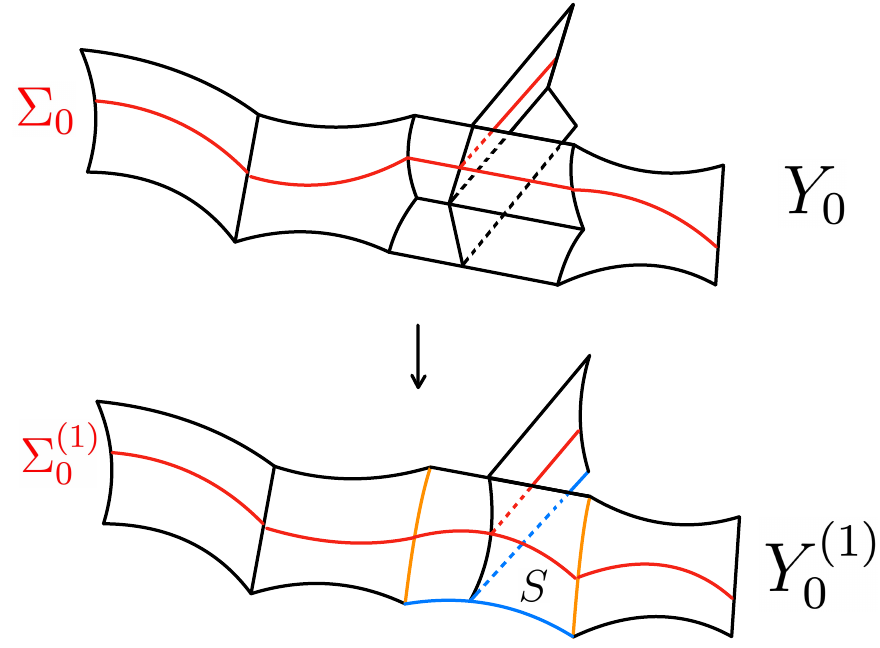}
    \caption{Log canonical model morphism $Y\rightarrow Y^{(1)}$. In blue, the curves to which some irreducible components of $Y_0$ are contracted to.}
    \label{fig:Y1 over C}
\end{figure}

\begin{corollary}\label{cor_after_the_contractions_we_still_have_P1_bundle_over_orb_curve}
Each irreducible component of $Y^{(1)}_0 \to C_0$ is the coarse space of a $\bP^1$–bundle over an orbifold curve.
\end{corollary}

\begin{proof}
Let $\sC \to C$ be the smooth canonical covering stack. Since $C$ is a family of nodal curves, its singularities are analytically of the form
\[
\Spec\bigl(k[\![x,y,t]\!]/(xy - t^n)\bigr).
\]
The canonical smooth covering stack replaces such a singularity with
\[
\bigl[\Spec(k[\![u,v]\!]) / \bmu_n\bigr],
\]
where $\bmu_n$ acts via $\zeta \cdot u = \zeta u$ and $\zeta \cdot v = \zeta^{-1} v$, and the coarse space map sends
\[
x \mapsto u^n,\qquad y \mapsto v^n,\qquad t \mapsto uv.
\] Away from finitely many points of $C_0$, the morphism $Y^{(1)} \to C$ is a $\bP^1$–fibration. Since $\sC$ and $C$ agree away from the nodes of $C_0$, there exists a big open subscheme $U \subseteq \sC$ together with a $\bP^1$–fibration over $U$. This produces a morphism
\[
U \longrightarrow \cB \PGL_2,
\]
and by \cite[Lemma~2.1]{twisted_map_2} this extends uniquely to a morphism
\[
\sC \longrightarrow \cB \PGL_2.
\]
Consequently, we obtain a $\bP^1$–fibration $\sY \to \sC$, and denote its coarse space by $Y' \to C$. Now, $Y^{(1)}$ carries a divisor--namely $\Sigma^{(1)}$--which is relatively ample over $C$. Its proper transform in $Y'$ gives a divisor $\Sigma'$ that is also relatively ample over $C$. As both $Y'$ and $Y^{(1)}$ are $S_2$, then by Lemma \ref{lem:isom in codim 1 implies isom}, $Y'$ is isomorphic to $Y^{(1)}$. In particular, each component of $Y^{(1)}_0$ is the coarse space of a $\bP^1$–bundle over an orbifold curve.
\end{proof}

We now run an MMP to decrease $\epsilon_2$ to $0$ in the pair
\[
\left(Y^{(1)},\; (\tfrac{1}{2}+\epsilon_1)\Sigma^{(1)} + \tfrac{1}{2}T^{(1)} + \epsilon_2 F\right).
\]
The canonical bundle formula for the lc-trivial fibration
\[
\left(Y^{(1)}, \tfrac{1}{2}\bigl(\Sigma^{(1)} + T^{(1)}\bigr)\right) \longrightarrow C
\]
yields
\[
K_{Y^{(1)}} + \tfrac{1}{2}\bigl(\Sigma^{(1)} + T^{(1)}\bigr)
\ \sim_{\bQ}\ 
\pi^{(1)*}(K_C + \bfM),
\]
where the discriminant divisor is trivial and $\bfM$ is the moduli divisor. Since 
\[
\left(Y,\tfrac{1}{2}(\Sigma+T)\right) \rightarrow \left(Y^{(1)}, \tfrac{1}{2}(\Sigma^{(1)} + T^{(1)})\right)
\]
is crepant, then these two pairs have the same moduli part; hence $K_C + \bfM$ is nef. To take the log canonical model of 
\[
\left(Y^{(1)}, (\tfrac{1}{2}+\epsilon_1)\Sigma^{(1)} + \tfrac{1}{2}T^{(1)} \right),
\]
one considers the log canonical model
\[
(C,\bfM) \longrightarrow (C^{\st},\bfM^{\st}),
\]
which contracts precisely:

\begin{enumerate}
    \item rational bridges $B \subseteq C_0$ along which $\deg_B(\bfM)=0$, and  
    \item rational tails $G$ of $C_0$ along which $\deg_G(\bfM)=1$.
\end{enumerate}

\smallskip

\noindent \textit{Rational bridges with $\deg_B(\bfM)=0$.}  
Let $B$ be such a rational bridge.  
By Corollary~\ref{cor_after_the_contractions_we_still_have_P1_bundle_over_orb_curve}, the restriction (see the component $S$ in Figure \ref{fig:Y1 over C})
\[
\left(Y^{(1)},(\tfrac{1}{2}+\epsilon_1)\Sigma^{(1)} + \tfrac{1}{2}T^{(1)}\right)\big|_B
\]
is birational to (see Figure \ref{fig:The pair})
\[
\left(\bP_B,\ (\tfrac{1}{2}+\epsilon_1)\Sigma_{\bP_B} + \tfrac{1}{2}\Sigma_2 + \Delta + F\right),
\]
where:
\begin{itemize}
    \item $\mtc{B}\to B\simeq\bP^1$ is a (possibly trivial) root stack;
    \item $\bP_B$ is the coarse space of a $\bP^1$–bundle $\bP_{\mtc{B}}(\mtc{E}) \to \mtc{B}$;
    \item $\Delta$ is the image in $Y^{(1)}$ of the horizontal double locus of $Y_0|_B$, i.e. where the two components of $Y|_B$ intersect;
    \item $\Sigma_{\bP_B}$ is a section of $\bP_B \to B$, and $\Sigma_2$ is determined by the relation  
    \[
    \tfrac{1}{2}T|_{Y^{(1)}_0} - \Delta - \tfrac{1}{2}\Sigma_{\bP_B} = \tfrac{1}{2}\Sigma_2;
    \]
    \item $F=F_1+F_2$ is the coarse space of the two fibers of $\bP_{\mtc{B}} \to \mtc{B}$.
\end{itemize}
\begin{figure}[H]
    \centering
    \includegraphics[width=0.4\linewidth]{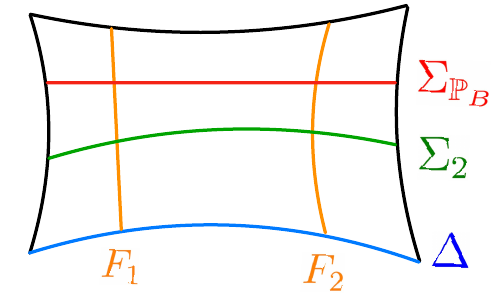}
    \caption{The pair $\left(\bP_B,\ \Sigma_{\bP_B} + \Sigma_2 + \Delta + F\right)$}
    \label{fig:The pair}
\end{figure}

\noindent
Observe that $\Sigma_{\bP_B}$, $\Sigma_2$, and $\Delta$ must be pairwise disjoint in order for 
$\deg \mathbf{M}|_B=0$. Indeed, $\Sigma_{\bP_B}$ and $\Sigma_2$ a re disjoint from Lemma \ref{lem:T and Sigma are disjoint}, and $\Sigma$ and $\Delta$ are disjoint as before the contraction $Y\to Y^{(1)}$ the section does not intersect the horizontal double locus of the central fiber. We just need to argue that $\Delta$ and $\Sigma_2$ are disjoint.
But if they met over a point $x$, the boundary part over that point would not be 0, as $\Delta$ has coefficient 1 and $\Sigma_2$ has coefficient $\frac{1}{2}$. The degree of the moduli part $\deg \mathbf{M}|_B$ agrees with the sum of the degree of the moduli part and the boundary part for $(Y^{(1)}|_B, \frac{1}{2}\Sigma_{\bP_B} + \frac{1}{2}\Sigma_2 + \Delta)$, so $\deg \mathbf{M}|_B>0$ which is a contradiction.

Consequently, when taking the canonical model of 
\[
\left(Y^{(1)},\left(\tfrac{1}{2}+\epsilon_1\right)\Sigma^{(1)}+\tfrac{1}{2}T^{(1)}\right),
\] the surface $Y^{(1)}|_B$ is contracted to the two fibers to which it is attached in neighboring components of $Y_0^{(1)}$ 

\medskip
\noindent \textit{Rational tails where $\mathbf{M}$ has degree~$1$.}
We now consider components of $Y^{(1)}_0$ lying over a rational tail $G$ of $C_0$ such that 
$\deg \mathbf{M}|_G=1$.

\begin{lemma}
The component $Y^{(1)}_G:=Y^{(1)}_0|_G$ is contracted by taking the log canonical model of 
\[
\left(Y^{(1)},(\tfrac{1}{2}+\epsilon_1)\Sigma^{(1)}+\tfrac{1}{2}T^{(1)}\right).
\]
\end{lemma}

\begin{proof}
By the classification of the components of $Y^{(1)}_0$, the pair $\left(Y^{(1)},(\tfrac{1}{2}+\epsilon_1)\Sigma^{(1)}+\tfrac{1}{2}T^{(1)}\right)\big|_G$ is of the form
\[
\left(\bP_G,(\tfrac{1}{2}+\epsilon_1)\Sigma^{(1)}|_{\bP_G} 
+\tfrac{1}{2}T^{(1)}|_{\bP_G} + F\right),
\]
where $\bP_G$ is the coarse space of a $\mathbb{P}^1$-bundle $\bP_\cG$ over a stacky curve $\cG$, and $F$ is the coarse space of one fiber over the unique point of $\cG$ lying above the node of $C_0$.
By assumption,
\begin{equation}\label{eq_lc_trivial}
    K_{\bP_G}+\tfrac{1}{2}\Sigma^{(1)}|_{\bP_G} 
    +\tfrac{1}{2}T^{(1)}|_{\bP_G} + F \sim_{\mathbb{Q}} 0.
\end{equation}
Thus the desired log canonical model of $\bP_G$ is also the ample model of $\Sigma^{(1)}|_{\bP_G}$.  
Since $\Sigma^{(1)}|_{\bP_G}$ is irreducible, it suffices to show that $(\Sigma^{(1)}|_{\bP_G})^2\le 0$.

{By Corollary \ref{cor_W_remains_lc_when_sigma_has_coef_1}, and since $\Sigma$ and $T$ don't intersect from Lemma \ref{lem:T and Sigma are disjoint}, the pair
\[
\left(\bP_G,\Sigma|_{\bP_G}+\tfrac{1}{2}T|_{\bP_G}+F\right)
\]
is lc, and hence the pair
\[
\left(\bP_\cG,\Sigma_{\bP_\cG}+\tfrac{1}{2}\mathcal{T}_{\bP_\cG}+\cF\right)
\]
is also lc}, where we use curly letters to denote the reduced preimages under the coarse space map 
$\phi\colon \bP_\cG\to\bP_G$.  The morphism $\phi$ is \'etale in codimension~$1$ away from the unique point over which the tail attaches, and otherwise ramified exactly along $F$, which appears with coefficient~$1$ in the boundary.  
Thus one has
\[
\phi^*(K_{\bP_G}+F)\sim_\bQ K_{\bP_\cG}+\cF,\qquad
\phi^*(\Sigma|_{\bP_G})\sim_\bQ \Sigma_{\bP_\cG},\qquad
\phi^*(T|_{\bP_G})\sim_\bQ \mathcal{T}_{\bP_\cG},
\] and consequently
\begin{align}\label{align_eq}
  \big(K_{\bP_G}+\Sigma|_{\bP_G}+F  .\  \Sigma|_{\bP_G}\big)
  &= \big(\phi^*(K_{\bP_G}+\Sigma|_{\bP_G}+F)  .\  \phi^*\Sigma|_{\bP_G}\big) \\
  &= \big(K_{\bP_\cG}+\Sigma_{\bP_\cG}+\cF . \ \Sigma_{\bP_\cG}\big) \\
  &= \deg\, \omega_{\Sigma_{\bP_\cG}}(p),
\end{align}
where $p=\cF\cap \Sigma_{\bP_\cG}$ is a single (stacky) point with reduced structure.  
If $\Sigma_{\bP_\cG}$ is an orbifold $\mathbb{P}^1$ with a single stacky point, and 
$\psi:\Sigma_{\bP_\cG}\to\mathbb{P}^1$ is the coarse map, then
\[
\omega_{\Sigma_{\bP_\cG}}(p)=\psi^*(\omega_{\mathbb{P}^1}(p)).
\] On the other hand, from \eqref{eq_lc_trivial},
\[
K_{\bP_G}+\tfrac{1}{2}\Sigma|_{\bP_G}+\tfrac{1}{2}T|_{\bP_G}+F\ 
\sim_\bQ \ \pi^*(\omega_{\mathbb{P}^1}(p)+\mathbf{M}|_G)\ \sim_\bQ\  0,
\]
where $p\in G$ is the attaching point.  
Since $\deg \mathbf{M}|_G\ge 0$, we obtain
\begin{align*}
\left(K_{\bP_G}+\tfrac{1}{2}\Sigma|_{\bP_G}
+\tfrac{1}{2}T|_{\bP_G}+F\right)\cdot\Sigma|_{\bP_G} &= 0,\\[4pt]
\left(K_{\bP_G}+\Sigma|_{\bP_G}
+\tfrac{1}{2}T|_{\bP_G}+F\right)\cdot\Sigma|_{\bP_G} &= -\deg(\mathbf{M}|_G)\le 0.
\end{align*}
Thus $(\Sigma|_{\bP_G})^2\le 0$, completing the proof.
\end{proof}

As a consequence, denoting by
\[
\left(Y^{(1)},(\tfrac{1}{2}+\epsilon_1)\Sigma^{(1)}+\tfrac{1}{2}T^{(1)}\right)
\ \longrightarrow\
\left(Y^{\st},(\tfrac{1}{2}+\epsilon_1)\Sigma^{\st}+\tfrac{1}{2}T^{\st}\right)
\]
the log canonical model morphism, we see that 
$\pi^{\st}:Y^{\st}\to C^{\st}$ has $1$-dimensional fibers and is a $\mathbb{P}^1$-fibration in codimension~$1$.  
Proceeding as in Corollary \ref{cor_after_the_contractions_we_still_have_P1_bundle_over_orb_curve}, and using \cite[Lemma~2.1]{twisted_map_2}, we obtain that $C^{\st}$ is the coarse moduli space of a family of twisted curves $\cC^{\st}$ admitting a $\mathbb{P}^1$-fibration 
$\cY^{\st}\to\cC^{\st}$ whose coarse space is $Y^{\st}\to C^{\st}$:
\[
\xymatrix{
\cY^{\st}\ar[r]\ar[d] & Y^{\st}\ar[d] \\
\cC^{\st}\ar[r] & C^{\st}.
}
\] Taking the double cover of $Y^{\st}$ as in Corollary \ref{cor_to_take_ksba_limit_one_can_take_quotient_first}, 
branched along $D^{\st}=\Sigma^{\st}+T^{\st}$, yields a family of elliptic surfaces
\[
(X^{\st},2\epsilon_1 S^{\st}) \longrightarrow C^{\st},
\]
where $S^{\st}$ is the reduced preimage of $\Sigma^{\st}$.  
Finally, the fibers of $Y\to C^{\st}$ over the smooth locus of $C^{\st}\to\operatorname{Spec}A$
are of the Weierstrass form $y^2 z = x^3 + Axz^2 + Bz^3$.  
It follows that all geometric fibers are of this form, again using \cite[Lemma~2.1]{twisted_map_2} and the fact that the good moduli space of $[\mathbb{A}^2/\mathbb{G}_m]$, the moduli of planar Weierstrass models, is a point; see \cite[Remark~7.10]{BFHIMZ24} for a similar argument.

\smallskip
This completes the proof of Theorem~\ref{thm_W_fib_with_a_section}.

\subsection{Applications of Theorem \ref{thm_W_fib_with_a_section}}

The first application of Theorem~\ref{thm_W_fib_with_a_section} is that it allows us to prove the following statement \emph{without running any explicit MMP}. Indeed, the result is a direct consequence of Theorem~\ref{thm_W_fib_with_a_section}.  

Let $\epsilon>0$ be a rational number, and $\overline{\mtc{W}}^{\KSBA}(\epsilon)$ denote the irreducible component of the KSBA moduli stack whose general point parametrizes a pair $(X,\epsilon S)$ admitting a Weierstrass fibration $(X,S)\longrightarrow C$ over a smooth curve $C$.

\begin{theorem}[\textup{cf. \cite[Theorem 1.2]{inchiostro_elliptic}}]\label{thm:weierstrass fibration in general}
     Let $0<\epsilon\ll1$ be a (rational) number. Then any pair $(X,\epsilon S)$ parametrized by $\ove{\mtc{W}}^{\KSBA}(\epsilon)$ admits a (not necessarily flat) fibration $$f\ :\ (X,\epsilon S)\ \longrightarrow \ C$$ to a nodal curve $C$ with purely $1$-dimensional fibers such that $$X|_{C^{\sm}}\ \longrightarrow \ C^{\sm}$$ is flat with integral fibers. Moreover, for any irreducible component $G$ of $C$ with normalization $G^\nu$, the surface $X|_{G^\nu}$ is one of the following:
    \begin{enumerate}
        \item a normal elliptic surface $X|_{G^\nu}\rightarrow G^\nu$, or
        \item a non-normal fibered surface $X|_{G^\nu}\rightarrow G^\nu$, whose general fiber is an integral nodal curve or arithmetic genus $1$.
    \end{enumerate}
\end{theorem}
In particular, the surface pairs appearing on the boundary of the KSBA moduli space 
parametrizing pairs
\[
\left(Y_\eta,\left(\tfrac{1}{2}+\epsilon_1\right)\Sigma_\eta
+\tfrac{1}{2}T_\eta+\left(\tfrac{1}{12}+\epsilon_2\right)F_\eta\right),
\]
which arise as degenerations of the $2{:}1$ quotients of elliptic surfaces with $j$-map of degree~$12$,
have KSBA-stable limits fibered over a family of nodal curves.
The base family is a stable generalized pair $(C,\mathbf{M}+\tfrac{1}{12}\mathbf{p})$, 
where $\mathbf{p}$ is a divisor of degree~$12$ and $\deg \mathbf{M}=1$.

The special fiber consists of a nodal curve
\[
C_0 = C_{0,1}\cup C_{0,2} 
\]
with marked points $p_1,\ldots,p_{12}$ and a non–negative $\mathbb{Q}$-divisor $\mathbf{M}_0$ such that
\[
K_{C_0} + \tfrac{1}{12}(p_1+\cdots+p_{12}) + \mathbf{M}_0
\]
is ample.
A direct combinatorial check shows that $C_0$ has at most two irreducible components; 
if it has two, then each component carries a $j$-map of degree~$6$ and exactly six of the marked points.

\medskip

Now increase the coefficient of the section from $\tfrac{1}{2}+\epsilon_1$ to $\tfrac{1}{2}+12\epsilon_2$.
For a general fiber, the pair
\[
\left(Y_\eta,\left(\tfrac{1}{2}+\epsilon_1\right)\Sigma_\eta
+\tfrac{1}{2}T_\eta+\left(\tfrac{1}{12}+\epsilon_2\right)F_\eta\right)
\]
is no longer KSBA-stable.
Nevertheless, for any KSBA-stable limit
\[
\left(Y,\left(\tfrac{1}{2}+\epsilon_1\right)\Sigma
+\tfrac{1}{2}T+\left(\tfrac{1}{12}+\epsilon_2\right)F\right)
\]
the log canonical divisor is nef.
Indeed, it suffices to check that for every irreducible component $\Gamma$ of the special fiber of the section~$\Sigma$,
\[
\left(K_Y + \left(\tfrac{1}{2}+\epsilon_1\right)\Sigma 
+\tfrac{1}{2}T+\left(\tfrac{1}{12}+\epsilon_2\right)F\right)\!\cdot\! \Gamma = 0.
\]

This follows from the equalities
\[
\left(K_Y + \tfrac{1}{2}\Sigma + \tfrac{1}{2}T + \tfrac{1}{12}F\right)\!\cdot\!\Gamma = 0, 
\qquad
T\!\cdot\!\Gamma =0,
\qquad
(K_Y+\Sigma)\!\cdot\!\Gamma = -1,
\]
where the first equality comes from the canonical bundle formula, 
and the latter two follow from the fact that $(Y,\Sigma)$ is slc and smooth along $\Sigma$ 
away from the nodes of the central fiber of $Y\to \operatorname{Spec} A$.

Thus, taking the canonical model under the specialization $\epsilon_1 = 12\epsilon_2$ contracts the section, and yields the following. In particular, the wall-crossing morphism of \cite{ascher2023wall, meng2023mmp} give a map $\ove{\mtc{W}}^{\KSBA}_{\textup{rat}}\big(\epsilon_1,\frac{1}{12}+\epsilon_2\big)\to \overline{\cM}^{\KSBA}$ to the KSBA moduli space with coefficient $\frac{1}{12}+\epsilon_2$ and given volume, which contracts the divisor with coefficient $\epsilon_1$.

\begin{theorem}[rational elliptic surfaces, cf. \textup{\cite[Theorem 1.1 (a), (b)]{AB_del_pezzo}}]\label{thm:rational surfaces} Let $0<\epsilon_1\ll \epsilon_2\ll1$ be two (rational) numbers. Then any pair $\big(X,\epsilon_1 S+(\frac{1}{12}+\epsilon_2)F\big)$ parametrized by $\ove{\mtc{W}}^{\KSBA}_{\textup{rat}}\big(\epsilon_1,\frac{1}{12}+\epsilon_2\big)$ admits a (not necessarily flat) fibration $$\textstyle f\ :\ \big(X,\epsilon_1 S+(\frac{1}{12}+\epsilon_2)F\big)\ \longrightarrow \ C$$ to a nodal curve $C$ with purely $1$-dimensional fibers satisfying that 
       \begin{itemize}
           \item $S$ is a section of $X\rightarrow C$;
           \item $F$ consists of some $f$-fibers over $C^{\sm}$; and
           \item $X|_{C^{\sm}} \rightarrow C^{\sm}$ is flat with integral fibers.
       \end{itemize} Any irreducible component $G$ of $C$ is isomorphic to $\bP^1$. If $C$ is irreducible, then $X|_{G}$ is 
    \begin{enumerate}
        \item either a normal elliptic surface $X|_{G}\rightarrow G$ with ADE singularities; or
        \item a non-normal fibered surface $X|_{G}\rightarrow G$, whose general fiber is an integral nodal curve or arithmetic genus $1$.
    \end{enumerate}
    If $C$ is reducible, then $C$ has two components $G_1,G_2$, and $X|_{G_i}\rightarrow G_i$ is a normal elliptic surface for $i=1,2$. Taking the canonical model for $\epsilon_1=12 \epsilon_2$ contracts the section, and gives the KSBA-stable limits of smooth degree 9 Del Pezzo surfaces with marking being the images of the 12 singular fibers.
\end{theorem}

Similarly, the following follows from Theorem \ref{thm_W_fib_with_a_section}.
\begin{theorem}[\textup{elliptic K3 surfaces, cf. \cite[Theorem 9.1.4]{Bru15}}]
     Let $0<\epsilon_1\ll \epsilon_2\ll1$ be two (rational) numbers. Then any pair $\big(X,\epsilon_1 S+\epsilon_2 F\big)$ parametrized by $\ove{\mtc{W}}^{\KSBA}_{\textup{K3}}\big(\epsilon_1,\epsilon_2\big)$ admits a (not necessarily flat) fibration $$\textstyle f\ :\ (X,\epsilon_1 S+\epsilon_2F)\ \longrightarrow \ C$$ to a nodal curve $C$ with purely $1$-dimensional fibers satisfying that 
       \begin{itemize}
           \item $S$ is a section of $X\rightarrow C$;
           \item $F$ consists of some $f$-fibers over $C^{\sm}$; and
           \item $X|_{C^{\sm}} \rightarrow C^{\sm}$ is flat with integral fibers.
       \end{itemize} Any irreducible component $C$ is isomorphic to $\bP^1$, and for any irreducible component $G$ of $C$, the surface $X|_{G^\nu}$ is one of the following:
    \begin{enumerate}
        \item a normal elliptic surface $X|_{G}\rightarrow G\simeq \bP^1$, or
        \item a non-normal fibered surface $X|_{G}\rightarrow G\simeq \bP^1$, whose general fiber is an integral nodal curve or arithmetic genus $1$.
        \end{enumerate}
\end{theorem}

\bibliographystyle{alpha}
\bibliography{citation}
\end{document}